\documentclass[reqno]{amsart}

\usepackage{amssymb,amsmath,amsthm,xcolor,enumerate,hyperref,refcheck}

       \usepackage{tikz-cd}
       \usetikzlibrary{positioning}

\allowdisplaybreaks

\newtheorem{thrm}{Theorem}[section]
\newtheorem{cor}[thrm]{Corollary}
\newtheorem{lem}[thrm]{Lemma}
\newtheorem{prop}[thrm]{Proposition}

\theoremstyle{definition}
\newtheorem{defn}[thrm]{Definition}

\newtheorem{rem}[thrm]{Remark}

\renewcommand{\iff}{\Leftrightarrow}

\newcommand{\impl}%
{%
	\Rightarrow
}

\newcommand{\id}{\mathrm{id}}

\newcommand{\cI}[1]%
{%
	\mathcal I{(#1)}%
}

\newcommand{\cIui}[1]%
{%
	{\mathcal I}_{ui}{(#1)}%
}

\newcommand{\cS}[1]%
{%
	\mathcal S{(#1)}%
}

\newcommand{\cD}%
{%
	\mathcal D%
}

\newcommand{\cU}[1]%
{%
	\mathcal U{(#1)}%
}

\newcommand{\cG}[1]%
{%
	\mathcal G{(#1)}%
}

\newcommand{\cM}[1]%
{%
	\mathcal M{(#1)}%
}

\newcommand{\dom}[1]%
{%
	\operatorname{\mathrm{dom}}{#1}%
}

\newcommand{\ran}[1]%
{%
	\operatorname{\mathrm{ran}}{#1}%
}

\newcommand{\End}[1]%
{%
	\operatorname{\mathrm{End}}{#1}%
}

\newcommand{\iend}[1]%
{%
	\operatorname{\mathit{end}}{#1}%
}

\newcommand{\iin}[1]%
{%
	\operatorname{\mathit{in}}{#1}%
}

\newcommand{\Ker}[1]%
{%
	\operatorname{\mathrm{Ker}}{#1}%
}

\newcommand{\iker}[1]%
{%
	\operatorname{\mathit{ker}}{#1}%
}

\newcommand{\m}{{}^{-1}}

\newcommand{\0}{\theta}
\newcommand{\e}{\varepsilon}

\begin{document}

\title[Twisted partial actions and extensions]{Twisted partial actions and extensions of semilattices of groups by groups}

\author{Mikhailo Dokuchaev}
\address{Insituto de Matem\'atica e Estat\'istica, Universidade de S\~ao Paulo,  Rua do Mat\~ao, 1010, S\~ao Paulo, SP,  CEP: 05508--090, Brazil}
\email{dokucha@gmail.com}

\author{Mykola Khrypchenko}
\address{Departamento de Matem\'atica, Universidade Federal de Santa Catarina, Campus Reitor Jo\~ao David Ferreira Lima, Florian\'opolis, SC,  CEP: 88040--900, Brazil}
\email{nskhripchenko@gmail.com}

\subjclass[2010]{Primary 20M30; Secondary 20M18, 16S35, 16W22.}
\keywords{Twisted partial action, extension, semilattice of groups, inverse semigroup}

\thanks{The first author was partially supported by  CNPq of Brazil,  Proc. 305975/2013--7, and by FAPESP of Brazil, Proc. 2015/09162--9. The second author was supported by FAPESP of Brazil, Proc. 2012/01554--7.}

\begin{abstract}
 We introduce the concept of an extension of a semilattice of groups $A$ by a group $G$ and describe all the extensions of this type which are equivalent to the crossed products $A*_\Theta G$ by twisted partial actions $\Theta$ of $G$ on $A$. As a consequence, we establish a one-to-one correspondence, up to an isomorphism, between twisted partial actions of groups on semilattices of groups and  so-called Sieben twisted modules 
over $E$-unitary inverse semigroups.
\end{abstract}

\maketitle

\norefnames
\nocitenames

\section*{Introduction}

R. Exel's  pioneering papers  \cite{E-1,E-2,E-3,E0,E1} on applications of partial actions and related notions to $C^*$-algebras influenced an intensive  research both in the theory of operator algebras and  in algebra (see \cite{D2}).  Among the most recent advances  we point out the use of partial actions and partial representations to dynamical systems of type $(m,n)$ in \cite{AraEKa},  to Carlsen-Matsumoto algebras of subshifts  \cite{DE2} and  to algebras related to separated graphs  \cite{AraE}. In the latter paper, as a byproduct,  a remarkable application to paradoxical decompositions was given.

The concept of a twisted partial group action, introduced in \cite{E0} and studied from a purely algebraic point of view in 
\cite{DES1} and \cite{DES2}, involves a function, which satisfies the $2$-cocycle identity in some restricted sense. This rose the problem of clarifying the kind of cohomology suiting it. Furthermore,  the notion of a partial homomorphism as a certain map from a group to a monoid,  on the one hand, and the theory of projective representations of semigroups developed by B. Novikov in \cite{Nov1} and \cite{Nov2}, on the other hand, inspired the idea of partial projective group representations, whose foundations where established in \cite{DN}, \cite{DN2} and \cite{DoNoPi}.   

In \cite{E1}  R. Exel defined an inverse semigroup $\cS G$ by means of generators and relations, which  governs  partial representations and partial actions of a group $G$.  Later  in \cite{KL}, using a fact by M. Szendrei \cite{Szendrei89},  J. Kellendonk and M. Lawson proved that $\cS G$ is isomorphic to the Birget-Rhodes prefix expansion   of $G$ \cite{BR1,BR2} (see also \cite{LawMarSte}).  The semigroup $\cS G $ is an essential  technical ingredient of the theory of partial projective group representations and it is   also crucial to deal with   the group cohomology based on partial actions. The latter was introduced and studied  in  \cite{DK}, the basic idea being to replace a $G$-module, i.\,e. a (global) action of a group $G$ on an abelian group,  by a partial $G$-module, which is a unital partial action of $G$ on a commutative semigroup $A.$ It follows from the definition that one may assume that $A$ is inverse, i.\,e. $A$ is a semilattice of abelian groups. Thus, when passing from partial actions of 
$G$ to actions  of $\cS G ,$ we obtain a relation between partial $G$-modules and  modules over inverse semigroups, and this way the cohomology of inverse semigroups comes into the picture. The latter was introduced in the pioneering paper by H. Lausch  \cite{Lausch}, with a more general approach given   by M.~Loganathan  in \cite{Loganathan81}, where H. Lausch's cohomology was derived  as the cohomology of a suitable  small category. Moreover, in the context of regular monoids M.~Loganathan's cohomology can be seen as  the cohomology of  J.~Leech \cite{Leech75} (see  \cite[Theorem 2.9]{Loganathan81}). M.~Loganathan's view, amongst other advantages,   allows  one to use   general facts from categories and homological algebra, in particular, when dealing with functorial properties of cohomology groups. Nevertheless, it turned out that the  category of partial $G$-modules is not abelian, so that  technically it is more appropriate for us to deal with  the down-to-earth  H.~Lausch's method.  In fact, our 
category is made up  of abelian ``pieces'', each of which is isomorphic to a category of modules over some inverse semigroup. These inverse semigroups run over  those  homomorphic images of  $\cS G,$ which are the max-generated $F$-inverse monoids $S$ (see \cite[p. 196]{Lawson2002})  whose maximum group image is $G.$ For these  monoids  \cite{DK} gives a rather  comprehensive  connection between the Lausch-Leech-Loganathan cohomology of $S$   and the partial cohomology of $G.$ Moreover, it was also clarified in \cite{DK} that the partial Schur multiplier of a group $G,$  which appeared in the theory of partial projective group representations, is a disjoint union of second cohomology groups with values in partial   (in general, non-trivial) $G$-modules. Influenced  by H.~Lausch's construction of a free module  \cite{Lausch},  it was possible to define free partial $G$-modules and obtain the partial cohomology groups via free resolutions \cite{DK}.

The present paper is stimulated by the desire to interpret the second partial cohomology groups in terms of extensions. In doing this it became clear, at some point, that we may abandon the restriction on a partial action to be unital, imposed by cohomology theory in \cite{DK}, covering thus a more general situation. The starting point is the concept of an extension of a semilattice of groups $A$ by a group $G$ (see Section~\ref{subsec-ext-A-G}) and the key example of an extension is the crossed product   $A\ast _\Theta G$ by a twisted partial action $\Theta$ of $G$ on $A$ (Section~\ref{subsec-tw-pact}). These are  related to the notion of an extension 
of $A$ by an inverse semigroup $S,$  that of a twisted $S$-module and the corresponding crossed product given in \cite{Lausch} (see  Sections~\ref{Extensions-of-A-by-S} and ~\ref{subsec-tw-S-mod}). In particular, given an extension $A \to U \to G$ we show in Proposition~\ref{prop-ext-A-G-to-ext-A-S} that there is a refinement $A \to U \to  S \to G$ such that $A\to U \to S$ is an extension of $A$ by $S,$ where $S$ is an $E$-unitary semigroup. In the case of crossed products the refinement of the extension $A \to A*_\Theta G \to G$ is  $A \to A*_\Theta G \to E(A)*_\0 G  \to G,$ where $E(A)$ is the semilattice of  idempotents of $A,$ and $\0 $ is the (non-twisted) partial action of $G$ on $E(A)$ obtained by the restriction of $\Theta $ (see Proposition~\ref{prop-A*_Theta-G-is-ext}). The epimorphism $ A*_\Theta G \to E(A)*_\0 G$ is easily seen to have  an order preserving transversal   $E(A)*_\0 G \to  A*_\Theta G.$ This is a  crucial feature  in our treatment, and the extensions   $A \to U \to G,$ whose 
refinements have this property, are called admissible.  Moreover, the extensions $A \to U \to  S $ of $A$ by an inverse semigroup $S,$ such that   $ U \to  S $ possesses  an order preserving transversal $S\to U,$ produce twisted $S$-modules structures on $A,$ whose twistings satisfy a normality condition, considered by N.~Sieben in \cite{Sieben98}, which is stronger than the one imposed by H.~Lausch in \cite{Lausch}. For this reason we call them Sieben twisted modules.

Given a  Sieben twisted $S$-module structure $\Lambda $ on $A,$ we construct in Section~\ref{sec-Theta^Lambda} a twisted partial action $\Theta $ of $ \cG S $ on $A,$ such that the crossed products $A*_\Lambda S$ and $A*_\Theta\cG S$ are equivalent extensions of $A$ by $\cG S$. Here  $ \cG S $ is the maximal group image of $S.$ This is used to prove in Theorem~\ref{thrm-adm-ext-equiv-crossed-prod} that  any admissible extension $A \to U \to G$ is equi\-va\-lent to some crossed product  extension  $A \to A*_\Theta G \to G.$

The passage from   twisted partial actions to Sieben twisted modules is done in Section~\ref{sec-Lambda^Theta}, resulting in the final fact Theorem~\ref{thrm-Theta<->Lambda}, which establishes  an equi\-va\-len\-ce preserving  one-to-one correspondence between twisted partial actions of groups on $A$ and Sieben twisted module structures  on $A$ over $E$-unitary inverse semigroups.

Being originated by the concept of a twisted partial group action on a $C^*$-algebra,  our cohomology theory and its interpretations may stimulate further algebraic investigations, as well as  $C^*$-algebraic developments, due to the fact that  group cohomology appears in the theory of operator algebras in various contexts, in particular,  when dealing with  projective isometric group representations, group actions on topological spaces and $C^*$-algebras, crossed products, locally principal ${\mathcal T}$-bundles related to locally compact transformation groups etc.  More specifically, one may expect Galois theoretic and $C^*$-theoretic  interpretations of low dimensional partial cohomology groups  in the sense of the exact sequences involving Picard groups and Brauer groups obtained in \cite{CHR} and  \cite{CroRaeWil}.  In addition, for the general case of locally compact groups instead of ordinary group cohomology the authors of  \cite{CroRaeWil} use the theory developed by C. C. Moore \cite{Moore} for 
actions on Polish modules. Thus it would be natural to extend  Moore's cohomology  of locally compact groups to the context of partial actions, and, most likely,  our treatment of  Sieben twisted modules in interaction with twisted partial actions   would find a natural application in the form of  appropriate analogues.

\section{Inverse semigroups and congruences on them}\label{sec-inv-sem}

Recall that a semigroup $S$ is called {\it regular}, if for any $s\in S$ there exists $s\m\in S$ (an {\it inverse of} $s$), such that $ss\m s=s$ and $s\m ss\m=s\m$. A regular semigroup $S$ is said to be {\it inverse}, if each $s\in S$ has a unique $s\m$. Inverse semigroups are precisely those regular semigroups, whose idempotents commute~\cite[Theorem 1.17]{Clifford-Preston-1}. Each inverse semigroup admits the natural partial order defined by $s\le t\iff s=et$ for some idempotent $e$ (or, equivalently, $s=tf$ for some idempotent $f$). The set of idempotents of $S$, for which we use the standard notation $E(S)$, is a (meet) semilattice under $\le$, more precisely, the meet $e\wedge f$ of $e,f\in E(S)$ is their product $ef\in E(S)$.

The binary relation $\sigma$ with $(s,t)\in\sigma\iff\exists u\le s,t$ (equivalently, $es=et$ for some $e\in E(S)$) is a congruence on an inverse semigroup $S$. It is the {\it minimum group congruence} on $S$ in the sense that $S/\sigma$ is a group and $\sigma\subseteq\rho$ for any congruence $\rho$, such that $S/\rho$ is a group. The quotient $\cG S=S/\sigma$ is called the {\it maximum group image} of $S$. An inverse semigroup is said to be {\it $E$-unitary}, whenever $(e,s)\in\sigma$ and $e\in E(S)$ imply that $s\in E(S)$ (equivalently, $e\le s\impl s\in E(S)$). $E$-unitary inverse semigroups can also be characterized by the property that $(s,t)\in\sigma\iff s\m t,st\m\in E(S)$ (see~\cite[Theorem 2.4.6]{Lawson}). An {\it $F$-inverse monoid} is an inverse semigroup, in which each $\sigma$-class possesses a maximum element (under $\le$). Each $F$-inverse monoid is $E$-unitary (see~\cite[Proposition 7.1.3]{Lawson}).

A {\it semilattice of groups} is a semigroup $A$ which can be represented as a disjoint union $A=\bigsqcup_{l\in L} A_l$, where $L$ is a semilattice, $A_l$ is a subgroup of $A$ and $A_l A_m\subseteq A_{l\wedge m}$ for all $l,m\in L$. Note that in this situation $A$ is inverse (see~\cite[Theorem~7.52]{Clifford-Preston-2}), $L$ can be chosen to be $E(A)$ and each $A_e$, $e\in E(A)$, is the maximal subgroup of $A$ containing $e$, i.\,e. $A_e=\{a\in A\mid aa\m=a\m a=e\}$. It follows that 
\begin{enumerate}
 \item $aa\m=a\m a$ for all $a\in A$,
 \item $E(A)\subseteq C(A),$\footnote{Regular semigroups with   property (ii)  are called {\it Clifford} semigroups (see~\cite[IV.2]{Howie} and~\cite[5.2]{Lawson}).}
\end{enumerate}
where $C(A)$ stands for the center of  $A.$ The first property is immediate, since both $aa\m$ and $a\m a$ are equal to the identity of the group component containing $a$. The second one is~\cite[Lemma 4.8]{Clifford-Preston-1}. 

It turns out that for an inverse semigroup $A$ each one of the conditions (i)--(ii) is equivalent to the fact that $A$ is a semilattice of groups. Indeed, (i) is simply the statement that $A$ is a disjoint union of groups and hence by the ``only if'' part of~\cite[Theorem 4.11]{Clifford-Preston-1} a semilattice of groups, while (ii) is (A) of~\cite[Theorem IV.2.1]{Howie} (see also~\cite[Theorem 5.2.12]{Lawson}). In particular, any commutative inverse semigroup is a semilattice of (abelian) groups.

We shall recall some basic facts about congruences on inverse semigroups, which one can find in~\cite{Clifford-Preston-2}. Given a congruence $\rho$ on an inverse semigroup $S$, the {\it kernel}~\cite[p. 60]{Clifford-Preston-2} of $\rho$ is $\iker\rho=E(S/\rho)$. Equivalently, a $\rho$-class belongs to $\iker\rho$ if and only if it contains an idempotent of $S$. By Theorem~7.38 from~\cite{Clifford-Preston-2} each congruence is uniquely determined by its kernel\footnote{Some authors by the Kernel of $\rho$ (denoted by $\Ker\rho$) mean the union of the $\rho$-classes from $\iker\rho$. Such a kernel itself does not determine $\rho$ uniquely, and for this one needs to know additionally the restriction of $\rho$ to the idempotents of $S$, the so-called {\it trace} of $\rho$ (see~\cite[5.1]{Lawson}).}. 

In particular, for a homomorphism $\varphi:S\to S'$ of inverse semigroups the congruence $\ker\varphi=\{(s,t)\in S^2\mid\varphi(s)=\varphi(t)\}$ (the ``usual'' kernel of $\varphi$) can be recovered from the set of classes of elements of $S$ which are mapped to idempotents under $\varphi$. The latter set is called the {\it kernel} of $\varphi$ in the context of inverse semigroups and denoted by $\iker\varphi$ to distinguish it from the congruence $\ker\varphi$. Thus, $\iker\varphi=\iker{(\ker\varphi)}$  and $\iker\rho=\iker{\rho^\natural}$, where $\rho^\natural$ is the natural epimorphism induced by the congruence $\rho$. 

A collection of disjoint subsets of $S$, which can be realized as the kernel of some (uniquely determined) congruence on $S$, is called a {\it kernel normal system}. There is a characterization of the kernel normal systems of an arbitrary inverse semigroup (see~\cite[p. 60]{Clifford-Preston-2}). 

We shall be interested in the kernel normal systems which are {\it kernels} of idempotent-separating (i.\,e. injective on idempotents) epimorphisms of inverse semigroups. Given such an epimorphism $\pi:S\to S'$ and $e'\in E(S')$, the pre\-ima\-ge $\pi\m(e')$ is an (inverse) subsemigroup of $S$ containing exactly one idempotent $e$, where $\pi(e)=e'$. Hence it is a group which we denote by $N_e$. Thus, $\iker\pi$ is the collection $\{N_e\}_{e\in E(S)}$ of subgroups of $S$, so that $\bigsqcup_{e\in E(S)}N_e=\pi\m(E(S'))$ is an inverse subsemigroup of $S$ and hence a semilattice of groups. Notice also that by $\pi\m(E(S'))$ one uniquely restores $\iker\pi$ as the family of group components of $\pi\m(E(S'))$.

Conversely, let $S$ be an arbitrary inverse semigroup. For each $e\in E(S)$ fix a subgroup $N_e$ of $S$ with identity $e$ and consider the collection $\mathcal N=\{N_e\}_{e\in E(S)}$. It turns out that $\mathcal N$ is a kernel normal system of $S$ exactly when $N=\bigsqcup_{e\in E(S)}N_e$ is a subsemigroup of $S$ satisfying $sNs\m\subseteq N$ for all $s\in S$ (see~\cite[Theorem 7.54]{Clifford-Preston-2}). Such kernel normal systems are called {\it group kernel normal systems}~\cite[p. 66]{Clifford-Preston-2}. In this case the corresponding natural epimorphism is idempotent-separating (and its {\it kernel} is $\mathcal N$). Indeed, the class of the induced congruence $\rho_{\mathcal N}$ containing $s\in S$ is $sN_{s\m s}=N_{ss\m}s$ (see~\cite[Theorem 7.55]{Clifford-Preston-2}). In particular, $\rho_{\mathcal N}^\natural(e)=N_e$ for $e\in E(S)$. More precisely, $(s,t)\in\rho_{\mathcal N}$ if and only if $t=us$ for some $u\in N_{ss\m}=N_{tt\m}$, or, equivalently, $t=sv$ for some $v\in N_{s\m s}=N_{t\m t}$. 
Moreover, such $u$ and $v$ are unique. For if $t=us$, then $u=uu\m u=u\cdot uu\m=uss\m=ts\m$. Similarly $t=sv$ implies $v=vv\m v=v\m v\cdot v=s\m sv=s\m t$. 

\section{\texorpdfstring{Extensions of $A$ by $S$}{Extensions of A by S}}\label{Extensions-of-A-by-S}
Let $A$ be a semilattice of groups and $S$ an inverse semigroup. An extension of $A$ by $S$ (in the sense of Lausch~\cite[p. 283]{Lausch}) is an inverse semigroup $U$ with a monomorphism $i:A\to U$ and an idempotent-separating epimorphism $j:U\to S$, such that $i(A)=j\m(E(S))$.\footnote{In other words $\iker j=\{i(A)_e\}_{e\in E(U)}$.} Notice that in this case $j|_{E(U)}$ is an isomorphism between $E(U)$ and $E(S)$. It is also easily seen that $i$ maps isomorphically $E(A)$ onto $E(U)$.

Recall from~\cite[p. 283]{Lausch} that any two extensions $A\overset{i}{\to}U\overset{j}{\to}S$ and $A\overset{i'}{\to}U'\overset{j'}{\to} S$ of $A$ by $S$ are called {\it equivalent} if there is a homomorphism $\mu:U\to U'$ such that the following diagram
 	\begin{align}\label{eq-comm-diag-equiv-ext}
     	\begin{tikzpicture}[node distance=1.5cm, auto]
     		\node (A) {$A$};
     		\node (U) [right of=A] {$U$};
     		\node (S) [right of=U] {$S$};
     		\node (A') [below of=A]{$A$};
     		\node (U') [below of=U] {$U'$};
     		\node (S') [below of=S] {$S$};
     		\draw[->] (A) to node {$i$} (U);
     		\draw[->] (U) to node {$j$} (S);
     		\draw[->] (A') to node {$i'$} (U');
     		\draw[->] (U') to node {$j'$} (S');
     		\draw[-,double distance=2pt] (A) to node {} (A');
     		\draw[->] (U) to node {$\mu$} (U');
     		\draw[-,double distance=2pt] (S) to node {} (S');
     	\end{tikzpicture}
 	\end{align}
commutes. Note that reflexivity and transitivity of the equivalence of extensions are obvious, but, {\it a priori}, it is not clear that this relation is symmetric. Its symmetry follows from the next lemma.

    \begin{lem}\label{lem-mu-is-iso}
        If~\eqref{eq-comm-diag-equiv-ext} commutes, then $\mu$ is an isomorphism. 
    \end{lem}
    \begin{proof}
        Suppose that $\mu(u)=\mu(v)$ for some $u,v\in U$. Then 
	\begin{equation}\label{eq-mu(uv^(-1))=mu(vv^(-1))}
		\mu(uv\m)=\mu(vv\m)
	\end{equation}
and hence, applying $j'$ and using $j=j'\circ\mu$ from~\eqref{eq-comm-diag-equiv-ext}, we get $j(uv\m)=j(vv\m)\in E(S)$. Therefore, $uv\m=i(a)$ and $vv\m=i(b)$ for some $a,b\in A$. Since $i'=\mu\circ i$, we obtain $\mu(uv\m)=i'(a)$ and $\mu(vv\m)=i'(b)$. In view of~\eqref{eq-mu(uv^(-1))=mu(vv^(-1))} and injectivity of $i'$ we conclude that $a=b$ and thus 
	\begin{equation}\label{eq-uv^(-1)=vv^(-1)}
		uv\m=vv\m. 
	\end{equation}
Similarly\footnote{or immediately from the facts that $j$ is idempotent-separating and $j(u)=j(v)$}
	\begin{equation}\label{eq-u^(-1)u=v^(-1)v}
		u\m u=v\m v.
	\end{equation}
Finally, using~\eqref{eq-uv^(-1)=vv^(-1)} and~\eqref{eq-u^(-1)u=v^(-1)v} we see that
	$$
		u=u(u\m u)=(uv\m) v=vv\m v=v.
	$$
This proves the injectivity of $\mu$. 

For the surjectivity take an arbitrary $u'\in U'$ and find $u\in U$ such that $j'(u')=j(u)=j'(\mu(u))$. Then
	\begin{equation}\label{eq-u'u'^(-1)=mu(uu^(-1))}
		u'u'\m=\mu(uu\m)
	\end{equation} 
due to the fact that $j'$ is idempotent-separating. Moreover, there exists $a\in A$ such that 
	\begin{equation}\label{eq-mu(u)u'=i'(a)}
		\mu(u)\m u'=i'(a)=\mu(i(a)), 
	\end{equation}  as $ \mu (u)\m u' \in j' \m (E(S)).$
Set $v=ui(a)$ and note from~\eqref{eq-u'u'^(-1)=mu(uu^(-1))} and~\eqref{eq-mu(u)u'=i'(a)} that
	$$
		\mu(v)=\mu(ui(a))=\mu(u)i'(a)=\mu(u)\mu(u)\m u'=\mu(uu\m)u'=u'u'\m u'=u'.
	$$
    \end{proof}

    \section{\texorpdfstring{Twisted $S$-modules}{Twisted S-modules}}\label{subsec-tw-S-mod}
    
    We recall some terminology and facts from~\cite{Lausch}. 

    \begin{defn}\label{defn-rel_inv}
        Let $A$ be a semilattice of groups. An endomorphism $\varphi$ of $A$ is called {\it relatively invertible}, whenever there exist $\bar\varphi\in\End A$ and $e_\varphi\in E(A)$ satisfying
	\begin{enumerate}
		\item $\bar\varphi\circ\varphi(a)=e_\varphi a$ and $\varphi\circ\bar\varphi(a)=\varphi(e_\varphi)a$ for any $a\in A$; 
		\item $e_\varphi$ is the identity of $\bar\varphi(A)$ and $\varphi(e_\varphi)$ is the identity of $\varphi(A)$.
	\end{enumerate}
    \end{defn}
    
   It is immediately seen directly from the definition that $\bar\varphi$ is relatively invertible with $\bar{\bar\varphi}=\varphi$ and $e_{\bar\varphi}=\varphi(e_\varphi)$.  The set of relatively invertible endomorphisms is denoted by $\iend A$.
    
	\begin{rem}\label{rem-rel_inv_weakened_(ii)}
	    It is enough to require in (ii) of Definition~\ref{defn-rel_inv} that $\varphi(e_\varphi)$ is the identity of $\varphi(A)$.
	\end{rem}
    \noindent  For suppose that $\varphi\in\End A$ satisfies (i) of Definition~\ref{defn-rel_inv} for some $\bar\varphi\in\End A$ and $e_\varphi\in E(A)$, such that $\varphi(e_\varphi a)=\varphi(a)$  for all $a\in A$. Then define $\tilde\varphi=e_\varphi\bar\varphi\in\End A$ and note that
    \begin{align*}
	       \tilde\varphi\circ\varphi(a)&=e_\varphi\bar\varphi\circ\varphi(a)=e_\varphi\cdot e_\varphi a=e_\varphi a,\\
	       \varphi\circ\tilde\varphi(a)&=\varphi(e_\varphi\bar\varphi(a))=\varphi(e_\varphi)\cdot\varphi(e_\varphi)a=\varphi(e_\varphi)a.
	\end{align*}
	Moreover, (ii) is obviously fulfilled with $\bar\varphi$ replaced by $\tilde\varphi$.
	
	\begin{rem}\label{rem-rel_inv_(ii)}
	    If $\varphi\in\End A$ satisfies (i) of Definition~\ref{defn-rel_inv} for some $\bar\varphi\in\End A$ and $e_\varphi\in E(A)$, then $\varphi'=\varphi(e_\varphi)\varphi\in\iend A$.
	\end{rem}
	\noindent Indeed, we see from (i) that
	    \begin{align*}
	       \bar\varphi\circ\varphi'(a)&=\bar\varphi\circ\varphi(e_\varphi a)=e_\varphi\cdot e_\varphi a=e_\varphi a,\\
	       \varphi'\circ\bar\varphi(a)&  = \varphi(e_\varphi)\cdot\varphi(e_\varphi)a=  \varphi '  (e_\varphi)a,
	    \end{align*}
so $\varphi'$ also satisfies (i) with $\bar\varphi'=\bar\varphi$ and $e_{\varphi'}=e_\varphi$. By Remark~\ref{rem-rel_inv_weakened_(ii)} it remains to verify that $\varphi'(e_\varphi)=\varphi(e_\varphi)\varphi(e_\varphi)=\varphi(e_\varphi)$ is the identity of $\varphi'(A)$, which is trivial.\\

Denote by $\cIui A$ the inverse semigroup of isomorphisms between unital ideals of a semigroup $A$ (see~\cite{DK}).

    \begin{prop}\label{prop-iend<->I_ui}
        The set $\iend A$ forms an (inverse) subsemigroup of $\End A$ isomorphic to $\cIui A$.
    \end{prop}
    \begin{proof}
        Given $\varphi\in\iend A$, we prove that the restriction $\varphi|_{e_\varphi A}$ maps isomorphically $e_\varphi A$ to $\varphi(e_\varphi)A$ with $\bar\varphi|_{\varphi(e_\varphi)A}$ being its inverse. In fact, we shall use only (i) of Definition~\ref{defn-rel_inv}. Indeed, $\varphi(e_\varphi A)\subseteq\varphi(e_\varphi)A$, $\bar\varphi(\varphi(e_\varphi)A)\subseteq e_\varphi^2A=e_\varphi A$ and
        \begin{align*}
	       \bar\varphi\circ\varphi(e_\varphi a)&=e_\varphi\cdot e_\varphi a=e_\varphi a,\\
	       \varphi\circ\bar\varphi(\varphi(e_\varphi)a)&=\varphi(e_\varphi)\cdot\varphi(e_\varphi)a=\varphi(e_\varphi)a
	    \end{align*}
for any $a\in A$.

    Conversely, let $\psi$ be an isomorphism between $eA$ and $fA$ for $e,f\in E(A)$. Defining $\varphi_\psi(a)=\psi(ea)$ and $\overline{\varphi_\psi}(a)=\psi\m(fa)$, we make sure that
        \begin{align*}
	       \overline{\varphi_\psi}\circ\varphi_\psi(a)&=\psi\m(f\psi(ea))=\psi\m\circ\psi(ea)=ea,\\
	       \varphi_\psi\circ\overline{\varphi_\psi}(a)&=\psi(e\psi\m(fa))=\psi\circ\psi\m(fa)=fa,
	    \end{align*}
    and that $\varphi_\psi(e)=\psi(e)=f$ is the identity of $\varphi_\psi(A)=\psi(eA)=fA$. In view of Remark~\ref{rem-rel_inv_weakened_(ii)} this proves that $\varphi_\psi\in\iend A$.
    
    Now   writing  $\phi = \varphi|_{e_\varphi A} $    for  $\varphi\in\iend A,$ we have  for all $a\in A$:
    $$
         \varphi_{\phi }(a)=\phi(e_\varphi a) =\varphi(e_\varphi a),
    $$
    which is $\varphi(a)$ by (ii) of Definition~\ref{defn-rel_inv}. Furthermore, given an isomorphism $\psi:eA\to fA$, we have $e_{\varphi_\psi}=e$, as showed above. Moreover, for arbitrary $a\in eA$:
    $$
        \varphi_\psi|_{eA}(a)=\psi(ea)=\psi(a).
    $$
    Thus, the correspondence $\varphi   \mapsto   \varphi|_{e_\varphi A}$ is one-to-one. It will be convenient to us to prove that its inverse 
    $\psi  \mapsto   \varphi_\psi$ is a homomorphism.
    
    Taking  isomorphisms $\psi:eA\to fA$ and $\psi':e'A\to f'A$, we know that
    \begin{align*}
        \dom{(\psi\circ\psi')}&=\psi'\m(\ran{\psi'}\cap\dom\psi)=\psi'\m(f'A\cap eA)=\psi'\m(ef'A)\\
        &=\psi'\m(ef')\psi'\m(f'A)=\psi'\m(ef')e'A=\psi'\m(ef')A.
    \end{align*}
    Therefore, 
    \begin{align*}
        \varphi_{\psi\circ\psi'}(a)=\psi\circ\psi'(\psi'\m(ef')a)&=\psi\circ\psi'(\psi'\m(ef')\cdot e'a)=\psi(ef'\psi'(e'a))\\
        &=\psi(e\psi'(e'a))=\varphi_\psi\circ\varphi_{\psi'}(a)
    \end{align*}  for any $a\in A.$
    \end{proof}
    
    \begin{cor}[Proposition 8.1 from \cite{Lausch}]\label{cor-inv_and_idemp_of_iend_A}
    \leavevmode
    \begin{enumerate}
        \item The inverse of $\varphi\in\iend A$ is $\bar\varphi\in\iend A$;
        \item each $\epsilon\in E(\iend A)$ has the form $\epsilon(a)=ea$ for a unique $e\in E(A)$.
    \end{enumerate}
    \end{cor}
    \noindent For the  partial isomorphism corresponding to $\bar{\varphi } $ is $\bar\varphi|_{\varphi(e_\varphi)}:\varphi(e_\varphi)A\to\bar\varphi\circ\varphi(e_\varphi)A=e_\varphi A$, which is inverse to $\varphi|_{e_\varphi A}$ in $\cIui A$, as we have seen in the proof of Proposition~\ref{prop-iend<->I_ui}. The second part follows from the straightforward facts that the idempotents of $\cIui A$ are the maps $\id_{eA}$, $e\in E(A)$, and that $eA=fA\iff e=f$.

\begin{rem}\label{lambda-iso}
Each  $\varphi \in \iend A$ respects the decompositions of the ideals $ e_{\varphi} A$ and $ \varphi(e_{\varphi}) A$ into 
the group components. More precisely, for any $f\in E(A),$ such that  
$f \leq e_{\varphi},$ one has that   $\varphi  (A_{ f}) =A_{\varphi (f)}.$
\end{rem}

\noindent Indeed, $\varphi (A_f) \subseteq A_{\varphi (f)}$ and $\bar\varphi (A_{\varphi(f)})\subseteq A_{\bar\varphi \circ\varphi(f)}=A_f, $ as    $\varphi$ and     $\bar\varphi $  are   endomorphisms.\\


Let $A\overset{i}{\to}U\overset{j}{\to}S$ be an extension of a semilattice of groups $A$ by an inverse semigroup $S$. Fix a map $\rho:S\to U$ with $j\circ\rho=\id_S$. Following Lausch, we call such a map a {\it transversal} of $j$. Then our observation about kernel normal systems yields
    \begin{lem}\label{lem-u=i(a)rho(s)}\footnote{This is Lemma~7.1 from~\cite{Lausch}. For its proof Lausch refers to~\cite{Preston54}, however, ~\cite{Preston54} does not contain this fact explicitly; the proof which we give here is obtained using ideas from~\cite{Preston54}.} 
        For any $u\in U$ there is a unique pair $(a,s)$ of $a\in A$ and $s\in S$, such that
        \begin{enumerate}
            \item $u=i(a)\rho(s)$;
            \item $i(aa\m)=\rho(s)\rho(s)\m$.
        \end{enumerate}
    \end{lem}
    \begin{proof}
        Taking $s=j(u)$, we see that $j(u)=j(\rho(s))$, i.\,e. $(u,\rho(s))\in \ker j$, so $u=v\rho(s)$ for some (uniquely defined) $v\in(\iker j)_{\rho(s)\rho(s)\m}=i(A)_{\rho(s)\rho(s)\m}$.  Let $a\in A$ be such that $i(a)=v.$ 
        
        It is enough to prove the uniqueness of $s$, because $a$ depends only on $u$ and $s.$ Assuming (i) and (ii) for some $a\in A$ and $s\in S$, we note that $j(i(a))$, being an idempotent, coincides with $j(i(a))j(i(a))\m=j(i(a)i(a)\m)=j(\rho(s)\rho(s)\m)=ss\m$. Therefore, 
$s=ss\m s=j(i(a))j(\rho(s))=j(u)$.
    \end{proof}

    From now on we shall assume as in~\cite{Lausch} that any transversal $\rho$ maps $E(S)$ to a subset of $E(U)$ (this can always be arranged, because $j\m(e)$ contains an idempotent for any $e\in E(S)$). In fact, this easily implies that $\rho|_{E(S)}$ is the isomorphism $(j|_{E(U)})\m$ and hence 
    \begin{align}\label{eq-alpha=i^(-1)rho_E(S)}
     \alpha=i\m\circ\rho|_{E(S)} 
    \end{align}
    maps isomorphically $E(S)$ onto $E(A)$. Notice also that
    \begin{align}\label{eq-rho(s)rho(s^(-1))}
     \rho(s)\rho(s)\m=\rho(ss\m)
    \end{align}
for any $s\in S$, as both $\rho(s)\rho(s)\m$ and $\rho(ss\m)$ are idempotents of $U$ which are mapped to the same $ss\m\in E(S)$ by $j$. 

Since $j(\rho(s)\rho(t))=j(\rho(st))$, there exists a unique $f(s,t)\in A$ such that
\begin{align}\label{eq-rho(s)rho(t)}
 \rho(s)\rho(t)=i(f(s,t))\rho(st)
\end{align}
and $i(f(s,t))\in i(A)_{\rho(st)\rho(st)\m}$ (see our observation about congruences given in Section~\ref{sec-inv-sem}). But $\rho(st)\rho(st)\m = \rho(stt\m s\m)$ by~\eqref{eq-rho(s)rho(s^(-1))}. Thus, we obtain a map $f:S^2\to A$ with $f(s,t)\in A_{\alpha(stt\m s\m)}$.
    
Without going into details we recall from~\cite[p. 291]{Lausch} that there is an idempotent-separating homomorphism $\nu:U\to\iend A$, $u\mapsto\nu_u$, defined by 
    \begin{align}\label{eq-defn_of_nu_u}
        \nu_u(a)=i\m(ui(a)u\m),
    \end{align}
$a\in A$. Then we set 
\begin{align}\label{eq-lambda=nu-rho}
 \lambda=\nu\circ\rho:S\to\iend A,\ s\mapsto\lambda_s.  
\end{align}

The following properties of the triple $(\alpha,\lambda,f)$ were established in~\cite{Lausch}:
	\begin{enumerate}
		\item $\lambda_e(a)=\alpha(e)a$ for all $e\in E(S)$ and $a\in A$;
		\item $\lambda_s(\alpha(e))=\alpha(ses\m)$ for all $s\in S$ and $e\in E(S)$;
		\item $\lambda_s\circ\lambda_t(a)=f(s,t)\lambda_{st}(a)f(s,t)\m$ for all $s,t\in S$ and $a\in A$;
		\item $f(se,e)=\alpha(ses\m)$ and $f(e,es)=\alpha(ess\m)$ for all $s\in S$ and $e\in E(S)$;
		\item $\lambda_s(f(t,u))f(s,tu)=f(s,t)f(st,u)$ for all $s,t,u\in S$.
	\end{enumerate}

Indeed, $\lambda_e$ is $\nu_{\rho(e)}$ by~\eqref{eq-lambda=nu-rho}, which is the multiplication by $i\m(\rho(e))$ in view of~\eqref{eq-defn_of_nu_u}. It remains to apply~\eqref{eq-alpha=i^(-1)rho_E(S)} to get (i). 

For (ii) observe that $\rho(s)\rho(e)\rho(s)\m=\rho(ses\m)$ as idempotents of $U$, whose images under $j$ coincide. Then by~\eqref{eq-alpha=i^(-1)rho_E(S)} and~\eqref{eq-defn_of_nu_u}--\eqref{eq-lambda=nu-rho}
    $$
        \lambda_s(\alpha(e))=\nu_{\rho(s)}(i\m(\rho(e)))=i\m(\rho(s)\rho(e)\rho(s)\m)=i\m(\rho(ses\m))=\alpha(ses\m).
    $$
    
To prove (iii), calculate using~\eqref{eq-rho(s)rho(t)}--\eqref{eq-lambda=nu-rho}
    $$
        \lambda_s\circ\lambda_t=\nu_{\rho(s)}\circ\nu_{\rho(t)}=\nu_{\rho(s)\rho(t)}=\nu_{i(f(s,t))\rho(st)}=\nu_{i(f(s,t))}\circ\nu_{\rho(st)}=\nu_{i(f(s,t))}\circ\lambda_{st}
    $$
and note that $\nu_{i(a)}$ is the conjugation by $a$ for any $a\in A$. 

It follows from $\rho(s)\rho(e)=\rho(s)\rho(e)\cdot\rho(e)$ and~\eqref{eq-rho(s)rho(t)} together with the uniqueness part of  Lemma~\ref{lem-u=i(a)rho(s)} that $f(s,e)=f(s,e)f(se,e)$. Since $f(s,e),f(se,e)\in A_{\alpha(ses\m)}$, the latter yields $f(se,e)=\alpha(ses\m)$. The second part of (iv) is proved similarly.

The proof of (v) is standard: we calculate $\rho(s)\rho(t)\rho(u)$ in two ways using the associativity of $U$ and~\eqref{eq-rho(s)rho(t)}--\eqref{eq-lambda=nu-rho}. On the one hand this equals
    $$
        \rho(s)\rho(t)\cdot\rho(u)=i(f(s,t))\rho(st)\rho(u)=i(f(s,t)f(st,u))\rho(stu),
    $$
and on the other hand this is
    \begin{align*}
        \rho(s)\cdot\rho(t)\rho(u)&=\rho(s)i(f(t,u))\rho(tu)\\
        &=\rho(s)\cdot\rho(s)\m\rho(s)\cdot i(f(t,u))\rho(tu)\\
        &=\rho(s)i(f(t,u))\rho(s)\m\cdot\rho(s)\rho(tu)\\
        &=i(\nu_{\rho(s)}(f(t,u))f(s,tu))\rho(stu)\\
        &=i(\lambda_s(f(t,u))f(s,tu))\rho(stu).
    \end{align*}
Here $\rho(s)\m\rho(s)$, being an idempotent, belongs to $i(A)$ and hence to $E(i(A))=i(E(A))\subseteq i(C(A))$, so it commutes with $i(f(t,u))$. Observing that $$f(s,t) f(st,u), \lambda _s(f(t,u)) f(s,tu) \in A_{\alpha (stuu\m t\m s\m )}=A_{\rho (stu) \rho(stu)\m},  $$ we obtain (v)  from the uniqueness part of  Lemma~\ref{lem-u=i(a)rho(s)}.

\begin{defn}\label{defn-twisted_S-mod}
 Let $S$ be an inverse semigroup. A {\it twisted $S$-module} is a semilattice of groups $A$ together with a triple $\Lambda=(\alpha,\lambda,f)$, where $\lambda$ is a map $S\to\iend A$, $\alpha$ is an isomorphism $E(S)\to E(A)$ and $f:S^2\to A$ is a map with $f(s,t)\in A_{\alpha(stt\m s\m)}$ satisfying the properties (i)--(v) listed above.
\end{defn}

\begin{rem}\label{rem-lambda_s(e_lambda_s)}
 Let $(\alpha,\lambda,f)$ be a twisted $S$-module structure on $A$. Then 
 \begin{align}\label{eq-lambda_s(e_lambda_s)}
  \alpha(ss\m)\lambda_s(a)=\lambda_s(a)
 \end{align}
 for all $s\in S$ and $a\in A$. 
\end{rem}
\noindent Indeed, since $\lambda_s$ is relatively invertible, then by (ii) of Definition~\ref{defn-rel_inv} there exists $e_{\lambda_s}\in E(A)$, such that $\lambda_s(e_{\lambda_s}a)=\lambda_s(a)$ for arbitrary $s\in S$ and $a\in A$. Taking $a=\alpha(s\m s)$, we get by (ii) of Definition~\ref{defn-twisted_S-mod}
\begin{align*}
 \alpha(ss\m)&=\lambda_s(e_{\lambda_s}\alpha(s\m s))=\lambda_s(\alpha(\alpha\m(e_{\lambda_s})s\m s))\\
 &=\alpha(s\alpha\m(e_{\lambda_s})s\m ss\m)=\alpha(s\alpha\m(e_{\lambda_s})s\m)=\lambda_s(e_{\lambda_s}).\\ 
\end{align*}

Notice that since $E(A)\subseteq C(A)$, then for each fixed $b\in A$ the map $\xi_b(a)=bab\m$ is an endomorphism called {\it inner}. Moreover, it is relatively invertible with $\overline{\xi_b}=\xi_{b\m}$ and $e_{\xi_b}=b\m b=bb\m$.

\begin{prop}\label{prop-equiv-tw-S-mod}
 Let $A\overset{i}{\to}U\overset{j}{\to}S$ be an extension, $\rho,\rho'$ transversals of $j$ and $\Lambda=(\alpha,\lambda,f)$, $\Lambda'=(\alpha',\lambda',f')$ the corresponding twisted $S$-module structures on $A$. Then
 \begin{enumerate}
  \item $\alpha'=\alpha$;
  \item $\lambda'_s=\xi_{g(s)}\circ\lambda_s$;
  \item $f'(s,t)g(st)=g(s)\lambda_s(g(t))f(s,t)$
 \end{enumerate}
for some function $g:S\to A$ with $g(s)\in A_{\alpha(ss\m)}$.
\end{prop}
\begin{proof}
Since $\rho|_{E(S)}=\rho'|_{E(S)}=j\m|_{E(S)}$, then (i) is explained by~\eqref{eq-alpha=i^(-1)rho_E(S)}.

As $j\circ\rho'(s)=j\circ\rho(s)$, there is $g(s)\in A$, such that $\rho'(s)=i(g(s))\rho(s)$ and $i(g(s)g(s)\m)=\rho(s)\rho(s)\m$, the latter being $\rho(ss\m)=i\circ\alpha(ss\m)$ by~\eqref{eq-rho(s)rho(s^(-1))} and~\eqref{eq-alpha=i^(-1)rho_E(S)}. So, $g(s)\in A_{\alpha(ss\m)}$.
 
In view of~\eqref{eq-defn_of_nu_u}--\eqref{eq-lambda=nu-rho} one has
\begin{align*}
 \lambda'_s(a)&=i\m(\rho'(s)i(a)\rho'(s)\m)\\
 &=i\m(i(g(s))\rho(s)i(a)\rho(s)\m i(g(s)\m))\\
 &=g(s)i\m(\rho(s)i(a)\rho(s)\m)g(s)\m\\
 &=\xi_{g(s)}\circ\lambda_s(a),
\end{align*}
whence (ii).

Using~\eqref{eq-rho(s)rho(t)}, observe that on the one hand
\begin{align}\label{eq-f'(s,t)g(st)}
 \rho'(s)\rho'(t)=i(f'(s,t))\rho'(st)=i(f'(s,t)g(st))\rho(st),
\end{align}
and on the other hand, taking into account the fact that $E(U)=i(E(A))\subseteq i(C(A))$, we see that
\begin{align}
 \rho'(s)\rho'(t)&=i(g(s))\rho(s)i(g(t))\rho(t)\notag\\
 &=i(g(s))\rho(s)\cdot \rho(s)\m\rho(s)\cdot i(g(t))\rho(t)\notag\\
 &=i(g(s))\cdot\rho(s)i(g(t))\rho(s)\m\cdot\rho(s)\rho(t)\notag\\
 &=i(g(s)\lambda_s(g(t))f(s,t))\rho(st).\label{eq-g(s)lambda_s(g(t))f(s,t)}
\end{align}
Now (iii) follows from~\eqref{eq-f'(s,t)g(st)}--\eqref{eq-g(s)lambda_s(g(t))f(s,t)}, the uniqueness part of Lemma~\ref{lem-u=i(a)rho(s)} and the observation that both $f'(s,t)g(st)$ and $g(s)\lambda_s(g(t))f(s,t)$ belong to $A_{\alpha(stt\m s\m)}$.
\end{proof}

\begin{rem}\label{rem-g(e)}
 The restriction of the map $g$ from Proposition~\ref{prop-equiv-tw-S-mod} to $E(S)$ coincides with $\alpha$.
\end{rem}
\noindent To see this, one substitutes $s=t=e\in E(S)$ into (iii) of Proposition~\ref{prop-equiv-tw-S-mod} and uses (i), (iv) of Definition~\ref{defn-twisted_S-mod}.

\begin{defn}\label{defn-equiv-tw-S-mod}
 Two twisted $S$-module structures $(\alpha,\lambda,f)$ and $(\alpha',\lambda',f')$ on $A$, for which (i)--(iii) of Proposition~\ref{prop-equiv-tw-S-mod} hold, will be called {\it equivalent}.
\end{defn}

\begin{rem}\label{rem-equiv-tw-S-mod}
 This defines an equivalence relation on the set of twisted $S$-module structures on $A$.
\end{rem}
\noindent For if $\lambda'_s=\xi_{g(s)}\circ\lambda_s$, then $\xi_{g(s)\m}\circ\lambda'_s=\xi_{\alpha(ss\m)}\circ\lambda_s=\lambda_s$ by~\eqref{eq-lambda_s(e_lambda_s)}. Multiplying (iii) of Proposition~\ref{prop-equiv-tw-S-mod} by $\lambda_s(g(t)\m)g(s)\m$ on the left and by $g(st)\m$ on the right, one gets
\begin{align*}
 f(s,t)g(st)\m=\lambda_s(g(t)\m)g(s)\m f'(s,t)=g(s)\m\lambda'_s(g(t)\m)f'(s,t).
\end{align*}
This proves symmetry, reflexivity being trivial by taking $g(s) = \alpha (s s\m)$ and using  (\ref{eq-lambda_s(e_lambda_s)}). For transitivity suppose additionally that $\Lambda'$ is equivalent to $\Lambda''$, and a map $h:S\to A$ determines this equivalence. Then
\begin{align*}
 \xi_{h(s)g(s)}\circ\lambda_s&=\xi_{h(s)}\circ\xi_{g(s)}\circ\lambda_s=\xi_{h(s)}\circ\lambda'_s=\lambda''_s,\\
 f''(s,t)h(st)g(st)&=h(s)\lambda'_s(h(t))f'(s,t)g(st)=h(s)\lambda'_s(h(t))g(s)\lambda_s(g(t)) f(s,t) \\
 &=h(s)g(s)\lambda_s(h(t))g(s)\m g(s)\lambda_s(g(t)) f(s,t)\\
& =h(s)g(s)\lambda_s(h(t)g(t)) f(s,t).
\end{align*}

\begin{lem}\label{lem-equiv-ext-with-the-same-Lambda}
 Let $A\overset{i}{\to}U\overset{j}{\to}S$ and $A\overset{i'}{\to}U'\overset{j'}{\to}S$ be equivalent extensions of $A$ by $S$ with $\mu:U\to U'$ being the corresponding isomorphism. If $\rho$ is a transversal of $j$, then $\mu\circ\rho$ is a transversal of $j'$ inducing the same twisted $S$-module structure on $A$.
\end{lem}
\begin{proof}
 Indeed, $j'\circ\rho'=j'\circ\mu\circ\rho=j\circ\rho=\id_S$. Denoting by $(\alpha,\lambda,f)$ and $(\alpha',\lambda',f')$ the twisted $S$-module structures coming from $(U,\rho)$ and $(U',\rho')$, respectively, one observes by~\eqref{eq-alpha=i^(-1)rho_E(S)} and~\eqref{eq-rho(s)rho(t)}--\eqref{eq-lambda=nu-rho} that
\begin{align*}
 \alpha'&=i'\m\circ\rho'|_{E(S')}=(\mu\circ i)\m\circ(\mu\circ\rho|_{E(S)})=i\m\circ\rho|_{E(S)}=\alpha,\\
 \lambda'_s(a)&=i'\m(\rho'(s)i'(a)\rho'(s)\m)=(\mu\circ i)\m\circ\mu(\rho(s)i(a)\rho(s)\m)\\
 &=i\m(\rho(s)i(a)\rho(s)\m)=\lambda_s(a),\\
 \rho'(s)\rho'(t)&=\mu(\rho(s)\rho(t))=\mu(i(f(s,t))\rho(st))=i'(f(s,t))\rho'(st),
\end{align*}
whence $f'(s,t)=f(s,t)$. 
\end{proof}

\begin{lem}\label{lem-ext-by-isomorphic-S}
 Given an extension $A\overset{i}{\to}U\overset{j}{\to}S$ of $A$ by $S$, a transversal $\rho$ of $j$ and an isomorphism $\nu:S\to S'$, the sequence $A\overset{i}{\to}U\overset{\nu\circ j}{\to}S'$ is an extension of $A$ by $S'$ and the map $\rho'=\rho\circ\nu\m$ is a transversal of $\nu\circ j$. Moreover, if $\Lambda = (\alpha,\lambda,f)$ and $\Lambda' = (\alpha',\lambda',f')$ are the corresponding twisted module structures on $A$, then $\Lambda=\Lambda'\circ\nu$ in the sense that
 \begin{align}\label{Lambda=Lambda'-circ-nu}
  \alpha=\alpha'\circ\nu|_{E(S)},\ \ \lambda=\lambda'\circ\nu,\ \ f=f'\circ(\nu\times\nu).
 \end{align}
\end{lem}
\begin{proof}
 Indeed, using the precise formulas~\eqref{eq-alpha=i^(-1)rho_E(S)} and~\eqref{eq-rho(s)rho(t)}--\eqref{eq-lambda=nu-rho}, we make sure that
\begin{align*}
 \alpha'\circ\nu|_{E(S)}&=i\m\circ\rho'\circ\nu|_{E(S)}=i\m\circ\rho\circ\nu\m\circ\nu|_{E(S)}=i\m\circ\rho|_{E(S)}=\alpha,\\
 \lambda'\circ\nu(a)&=i\m(\rho'(\nu(s))i(a)\rho'(\nu(s))\m)=i\m(\rho(s)i(a)\rho(s)\m)=\lambda_s(a),\\
 \rho'(\nu(s))\rho'(\nu(t))&=\rho(s)\rho(t)=i(f(s,t))\rho(st)=i(f(s,t))\rho'(\nu(s)\nu(t)),
\end{align*}
and consequently $f(s,t)=f'(\nu(s),\nu(t))$.
\end{proof}

\begin{cor}\label{cor-equiv-ext-by-isomorphic-S}
 Let $A\overset{i}{\to}U\overset{j}{\to}S$ and $A\overset{i'}{\to}U'\overset{j'}{\to}S'$ be extensions, such that there exist isomorphisms $\mu:U\to U'$ and $\nu:S\to S'$ making the following diagram
 $$
      	\begin{tikzpicture}[node distance=1.5cm, auto]
      		\node (A) {$A$};
      		\node (U) [right of=A] {$U$};
      		\node (S) [right of=U] {$S$};
      		\node (A') [below of=A]{$A$};
      		\node (U') [below of=U] {$U'$};
      		\node (S') [below of=S] {$S'$};
      		\draw[->] (A) to node {$i$} (U);
      		\draw[->] (U) to node {$j$} (S);
      		\draw[->] (A') to node {$i'$} (U');
      		\draw[->] (U') to node {$j'$} (S');
      		\draw[-,double distance=2pt] (A) to node {} (A');
      		\draw[->] (U) to node {$\mu$} (U');
      		\draw[->] (S) to node {$\nu$} (S');
      	\end{tikzpicture}
  $$
 commute. If $\rho$ and $\rho'$ are transversals of $j$ and $j'$, respectively, with $\mu\circ\rho=\rho'\circ\nu$, then the corresponding twisted module structures $\Lambda=(\alpha,\lambda,f)$ and $\Lambda'=(\alpha',\lambda',f')$ on $A$ satisfy $\Lambda=\Lambda'\circ\nu$.  
\end{cor}
\noindent For by Lemma~\ref{lem-ext-by-isomorphic-S} the sequence $A\overset{i}{\to}U\overset{\nu\circ j}{\to}S'$ is an extension and $\rho''=\rho\circ\nu\m$ is a transversal of $\nu\circ j$. Moreover, the twisted module $\Lambda''$ coming from this extension can be obtained from $\Lambda$ by means of $\Lambda=\Lambda''\circ\nu$. Since $A\overset{i}{\to}U\overset{\nu\circ j}{\to}S'$ is equivalent to $A\overset{i'}{\to}U'\overset{j'}{\to}S'$ with $\mu$ being the equivalence and $\mu\circ\rho''=\mu\circ\rho\circ\nu\m=\rho'$, then $\Lambda''=\Lambda$ thanks to Lemma~\ref{lem-equiv-ext-with-the-same-Lambda}. Thus, $\Lambda=\Lambda'\circ\nu$.

\begin{defn}\label{defn-A*_Lambda-S}
 Let $\Lambda=(\alpha,\lambda,f)$ be a twisted $S$-module structure on $A$. By the {\it crossed product of $A$ and $S$ by $\Lambda$}~\cite{Lausch} we mean the set 
 $$
  A*_\Lambda S=\{a\delta_s\mid a\in A,\ \ s\in S,\ \ aa\m=\alpha(ss\m)\}.
 $$
 It is an inverse semigroup under the multiplication $a\delta_s\cdot b\delta_t=a\lambda_s(b)f(s,t)\delta_{st}$ with
 \begin{align}\label{eq-adelta_s-inv}
  (a\delta_s)\m=f(s\m,s)\m\lambda_{s\m}(a\m)\delta_{s\m}. 
 \end{align}
 Moreover, $A*_\Lambda S$ is an extension of $A$ by $S$, where $i(a)=a\delta_{\alpha\m(aa\m)}$ and $j(a\delta_s)=s$ (see~\cite[Theorem 9.1]{Lausch}).
\end{defn}

The next remark is explained by~\cite[Theorem 9.2]{Lausch}.
\begin{rem}\label{rem-A*_Lambda-S-equiv-U}
 Let $A\overset{i}{\to}U\overset{j}{\to}S$ be an extension, $\rho$ a transversal of $j$ and $\Lambda$ the corresponding twisted $S$-module structure on $A$. Then $\mu:A*_\Lambda S\to U$, $\mu(a\delta_s)=i(a)\rho(s)$, defines an equivalence of extensions $A*_\Lambda S$ and $U$.
\end{rem}

\begin{prop}\label{prop-rho-for-A*_Lambda-S} 
 The epimorphism $j:A*_\Lambda S\to S$ admits the natural transversal $\rho(s)=\alpha(ss\m)\delta_s$, such that the corresponding twisted $S$-module is $(\alpha,\lambda',f)$, where
 \begin{align}\label{eq-lambda'-for-A*_Lambda-S}
  \lambda'_s(a)=\lambda_s(a)f(s,s\m)\m f(s,\alpha\m(aa\m))f(s\alpha\m(aa\m),s\m).
 \end{align}
\end{prop}
\begin{proof}
 Clearly, $j\circ\rho(s)=j(\alpha(ss\m)\delta_s)=s$. Denote by $(\alpha',\lambda',f')$ the twisted $S$-module defined by $(A*_\Lambda S,\rho)$. According to~\eqref{eq-alpha=i^(-1)rho_E(S)}:
$$
 \alpha'(e)=i\m\circ\rho(e)=i\m(\alpha(e)\delta_e)=\alpha(e).
$$

Thanks to (ii) of Definition~\ref{defn-twisted_S-mod} and the fact that $f(s,t)\in A_{\alpha(stt\m s\m)}$:
\begin{align*}
\rho(s)\rho(t)&=\alpha(ss\m)\delta_s\cdot\alpha(tt\m)\delta_t=\alpha(ss\m)\lambda_s(\alpha(tt\m))f(s,t)\delta_{st}\\
 &=\alpha(stt\m s\m)f(s,t)\delta_{st}=f(s,t)\rho(st), 
\end{align*}
so $f'(s,t)=f(s,t)$.

For $\lambda'_s$ we need to calculate $\rho(s)\m$, as one sees from~\eqref{eq-defn_of_nu_u}--\eqref{eq-lambda=nu-rho}. By~\eqref{eq-adelta_s-inv}, (ii) of Definition~\ref{defn-twisted_S-mod} and the fact that $f(s\m,s)\in A_{\alpha(s\m s)}$:
\begin{align*}
 \rho(s)\m&=(\alpha(ss\m)\delta_s)\m=f(s\m,s)\m\lambda_{s\m}(\alpha(ss\m))\delta_{s\m}\\
 &=f(s\m,s)\m\alpha(s\m s)\delta_{s\m}=f(s\m,s)\m\delta_{s\m}.
\end{align*}
Then by (i) of Definition~\ref{defn-twisted_S-mod}
\begin{align*}
  i(a)\rho(s)\m&=a\delta_{\alpha\m(aa\m)}\cdot f(s\m,s)\m\delta_{s\m}\\
  &=a\lambda_{\alpha\m(aa\m)}(f(s\m,s)\m)f(\alpha\m(aa\m),s\m)\delta_{\alpha\m(aa\m)s\m}\\
  &=a\cdot aa\m f(s\m,s)\m f(\alpha\m(aa\m),s\m)\delta_{\alpha\m(aa\m)s\m}\\
  &=af(s\m,s)\m f(\alpha\m(aa\m),s\m)\delta_{\alpha\m(aa\m)s\m}.
\end{align*}
Hence $\lambda'_s(a)=i\m(\rho(s)i(a)\rho(s)\m)$ equals
\begin{align}\label{eq-lambda'_s(a)-not-simplified}
 \alpha(ss\m)\lambda_s(a)\lambda_s(f(s\m,s))\m\lambda_s(f(\alpha\m(aa\m),s\m))f(s,\alpha\m(aa\m)s\m).
\end{align}
The idempotent $\alpha(ss\m)$ can be removed from~\eqref{eq-lambda'_s(a)-not-simplified}, as $\lambda_s(f(s\m,s))\in A_{\alpha(ss\m)}$. Furthermore, it follows from (v) of Definition~\ref{defn-twisted_S-mod} with $t=s\m$ and $u=s$ that
$$
 \lambda_s(f(s\m,s))f(s,s\m s)=f(s,s\m)f(ss\m,s).
$$
But $f(s,s\m s)=f(ss\m,s)=\alpha(ss\m)$ by (iv) of the same definition. Since both $f(s,s\m)$ and $\lambda_s(f(s\m,s))$ belong to $A_{\alpha(ss\m)}$, one concludes that
\begin{align}\label{eq-lambda_s(f(s^(-1),s))=f(s,s^(-1))}
 \lambda_s(f(s\m,s))=f(s,s\m).
\end{align}
Now applying (v) of Definition~\ref{defn-twisted_S-mod} to the triple $(s,\alpha\m(aa\m),s\m)$, we see that
\begin{align}\label{eq-2-cocycle-for-(s,alpha^(-1)(aa^(-1)),s^(-1))}
 \lambda_s(f(\alpha\m(aa\m),s\m))f(s,\alpha\m(aa\m)s\m)=f(s,\alpha\m(aa\m))f(s\alpha\m(aa\m),s\m).
\end{align}
Substituting~\eqref{eq-lambda_s(f(s^(-1),s))=f(s,s^(-1))} and~\eqref{eq-2-cocycle-for-(s,alpha^(-1)(aa^(-1)),s^(-1))} into~\eqref{eq-lambda'_s(a)-not-simplified} we finally obtain~\eqref{eq-lambda'-for-A*_Lambda-S}.
\end{proof}

\begin{lem}\label{lem-lambda_s_satisfies_(i)}
 If in Definition~\ref{defn-twisted_S-mod} we replace the condition $\lambda_s\in\iend A$ by the weaker one $\lambda_s\in\End A$, then $\lambda_s$ still satisfies (i) of Definition~\ref{defn-rel_inv} with $\overline{\lambda_s}=\xi_{f(s\m,s)\m}\circ\lambda_{s\m}$ and $e_{\lambda_s}=\alpha(s\m s)$. 
\end{lem}
\begin{proof}
By (i)--(iii) of Definition~\ref{defn-twisted_S-mod}
\begin{align*}
\overline{\lambda_s}\circ\lambda_s(a)&=\xi_{f(s\m,s)\m}(\lambda_{s\m}(\lambda_s(a)))\\
&=\xi_{f(s\m,s)\m}(\xi_{f(s\m,s)}\circ\lambda_{s\m s}(a))\\
&=\xi_{f(s\m,s)\m f(s\m,s)}\circ\lambda_{s\m s}(a)\\
&=\xi_{\alpha(s\m s)}(\alpha(s\m s)a)\\
&=\alpha(s\m s)a,
\end{align*}
so $\alpha(s\m s)$ is a candidate for $e_{\lambda_s}$.

Now in view of (i) and (iii) of Definition~\ref{defn-twisted_S-mod}
$$
\lambda_s\circ\overline{\lambda_s}=\xi_{\lambda_s(f(s\m,s)\m)f(s,s\m)\alpha(ss\m)}
=\xi_{\lambda_s(f(s\m,s)\m)f(s,s\m)},
$$
because $f(s,s\m)\in A_{\alpha(ss\m)}$. It remains to notice by~\eqref{eq-lambda_s(f(s^(-1),s))=f(s,s^(-1))} that
$$
 \lambda_s(f(s\m,s)\m)f(s,s\m)=f(s,s\m)\m f(s,s\m)=\alpha(ss\m),
$$
so $\lambda_s\circ\overline{\lambda_s}(a)=\alpha(ss\m)a=\lambda_s(\alpha(s\m s))a$. 
\end{proof}

\begin{rem}\label{lambda-bar}  
If $\Lambda $ is a twisted $S$-module structure on $A,$ then $\overline{\lambda_s}$ given in Lemma~\ref{lem-lambda_s_satisfies_(i)} is the inverse of $\lambda _s$ in $\iend A$ with $e_{\lambda_s}=\alpha(s\m s)$ (see Definition~\ref{defn-rel_inv}).  
\end{rem}

\noindent Indeed,  we need to show that $\lambda _s$ and $\overline{\lambda_s}=\xi_{f(s\m,s)\m}\circ\lambda_{s\m}$ satisfy (i) and (ii) of Definition~\ref{defn-rel_inv} with $e_{\lambda_s}=\alpha(s\m s).$ Item (i) is satisfied by Lemma~\ref{lem-lambda_s_satisfies_(i)}.  Moreover, $\lambda _s(e_{\lambda_s})=\lambda_s(\alpha(s\m s))=\alpha(ss\m)$ is the identity of $\lambda_s(A)$ thanks to Remark~\ref{rem-lambda_s(e_lambda_s)}. Therefore, $e_{\lambda_s}\overline{\lambda_s}=\alpha(s\m s)\overline{\lambda_s}=\overline{\lambda_s}$ is the inverse of $\lambda_s$, as it was explained in the proof of Remark~\ref{rem-rel_inv_weakened_(ii)}.

    \begin{rem}\label{rem-lambda_s_rel_inv}
        If under the conditions of Lemma~\ref{lem-lambda_s_satisfies_(i)} we assume that $S$ is a monoid, then $\alpha(ss\m)$ is the identity of $\lambda_s(A)$. In particular, $\lambda_s\in\iend A$ and hence $\Lambda$ is a twisted $S$-module structure on $A$.
    \end{rem}
    \noindent Indeed, $\alpha(1_S)$, being $1_{E(A)}$, is $1_A$, because $1_{E(A)}a=1_{E(A)}aa\m\cdot a=aa\m\cdot a=a$ for any $a\in A$. Hence by (i) of Definition~\ref{defn-twisted_S-mod} the endomorphism $\lambda_1$ is trivial, so thanks to (iii)--(iv) of the same definition for any $s\in S$:
    $$
        \lambda_s=\lambda_s\circ\lambda_1=\xi_{f(s,1)}\circ\lambda_{s\cdot 1}=\xi_{\alpha(ss\m)}\circ\lambda_s.
    $$
    In other words, $\lambda_s(a)=\alpha(ss\m)\lambda_s(a)=\lambda_s(\alpha(s\m s))\lambda_s(a)$ and hence in view of Lemma~\ref{lem-lambda_s_satisfies_(i)} and Remark~\ref{rem-rel_inv_weakened_(ii)} the endomorphism $\lambda_s$ is relatively invertible.\\
    
    Item (c) of~\cite[Definition 2.2]{Sieben98} inspires us to consider twisted $S$-modules $(\alpha,\lambda,f)$, whose $f$ satisfies a normality condition, which is stronger than (iv) of Definition~\ref{defn-twisted_S-mod}, namely
    \begin{enumerate}
     \item[(iv')] $f(s,e)=\alpha(ses\m)$ and $f(e,s)=\alpha(ess\m)$ for all $s\in S$ and $e\in E(S)$.
    \end{enumerate}
    We call such twisted $S$-modules {\it Sieben} twisted $S$-modules (see also~\cite[Definition 5.4]{Buss-Exel11}). Using the idea of~\cite[Proposition 2.5]{Sieben98}, we describe those $\rho:S\to U$, which induce this class of twisted $S$-modules.
    
    \begin{prop}\label{prop-order-pres-rho}
     Let $A\overset{i}{\to}U\overset{j}{\to}S$ be an extension, $\rho$ a transversal of $j$ and $\Lambda=(\alpha,\lambda,f)$ the corresponding twisted module $S$-structure on $A$. Then $\Lambda$ is a Sieben twisted $S$-module structure if and only if $\rho$ is order-preserving, i.\,e. $s\le t\impl\rho(s)\le\rho(t)$ for all $s,t\in S$.
    \end{prop}
    \begin{proof}
     Suppose that $\rho$ preserves the order. Then by~\eqref{eq-rho(s)rho(t)}
     \begin{align*}
      \rho(e)\rho(s)=i(f(e,s))\rho(es)\le i(f(e,s))\rho(s).
     \end{align*}
     Let us multiply this on the right by $\rho(s)\m$. In view of~\eqref{eq-rho(s)rho(s^(-1))} we get
     \begin{align*}
      \rho(ess\m)\le i(f(e,s))\rho(ss\m).
     \end{align*}
     Using~\eqref{eq-alpha=i^(-1)rho_E(S)}, we rewrite this as 
     $$
      \alpha(ess\m)\le f(e,s)\alpha(ss\m).
     $$
    Therefore, 
    $$
     \alpha(ess\m)=\alpha(ess\m)f(e,s)\alpha(ss\m)=\alpha(ess\m)f(e,s),
    $$
    the latter being $f(e,s)$, as $f(e,s)\in A_{\alpha(ess\m)}$. The equality $f(s,e)=\alpha(ses\m)$ is proved similarly.
    
    Conversely, if $f(e,s)=\alpha(ess\m)$, then using~\eqref{eq-alpha=i^(-1)rho_E(S)}--\eqref{eq-rho(s)rho(t)} again, we see
    $$
     \rho(e)\rho(s)=i\circ\alpha(ess\m)\rho(es)=\rho(ess\m)\rho(es)=\rho(es(es)\m)\rho(es)=\rho(es).
    $$
    Thus, $\rho(es)\le\rho(s)$. It remains to note that each $t\le s$ has the form $es$ for some $e\in E(S)$.
    \end{proof}
    
    \begin{rem}\label{rem-f(e,s)-and-f(s,e)}
     It can be seen from the proof of Proposition~\ref{prop-order-pres-rho} that $f(e,s)=\alpha(ess\m)\iff f(s,e)=\alpha(ses\m)$.
    \end{rem}
    
    \begin{rem}\label{rem-equiv-ext-Sieben}
     The property of $j$ ``to have an order-preserving transversal'' is invariant under the equivalence of extensions.
    \end{rem}
    \noindent For if $\mu:U\to U'$ determines the equivalence of $U$ and $U'$, and $\rho$ is an order-preserving transversal of $j$, then $\mu\circ\rho$ is an order-preserving transversal of $j'$.\\
    
    The next proposition is inspired by~\cite[Proposition 2.8]{Sieben98}.
    \begin{prop}\label{prop-ext-by-F-inverse}
     Let $A\overset{i}{\to}U\overset{j}{\to}S$ be an extension and $S$ an $F$-inverse monoid. Then there exists an order-preserving transversal $\rho$ of $j$.
    \end{prop}
    \begin{proof}
      We know that any transversal of $j$ coincides with $(j|_{E(U)})\m$ on $E(S)$. So, our $\rho$ is uniquely defined on the idempotents of $S$, which constitute a whole $\sigma$-class of $S$, as $S$ is $F$-inverse. Fix a $\sigma$-class $x$ of $S$ distinct from $E(S)$. We shall construct the desired $\rho$ on $x$. Due to the fact that $x$ is arbitrary, this will define $\rho$ on the whole $S$. 
      
      Let $\max x$ denote the greatest element of $x$ and $e_x=\max x(\max x)\m$. Set $\rho(\max x)$ to be an arbitrary element of $j\m(\max x)$, so that $j(\rho(\max x))=\max x$. Since $j$ is idempotent-separating, $\rho(\max x)\rho(\max x)\m=\rho(e_x)$. Then 
      \begin{align}\label{eq-rho(s)=rho(ss^(-1))rho(s)}
       \rho(\max x)=\rho(e_x)\rho(\max x).
      \end{align}
      Now given $e\max x\in x$, $e\in E(S)$, define $\rho(e\max x)=\rho(e)\rho(\max x)$. Observe that if $e\max x=f\max x$, then $ee_x=fe_x$, and hence $\rho(e)\rho(e_x)=\rho(f)\rho(e_x)$. Multiplying this on the right by $\rho(\max x)$ and using~\eqref{eq-rho(s)=rho(ss^(-1))rho(s)}, we get $\rho(e)\rho(\max x)=\rho(f)\rho(\max x)$, so the definition makes sense. Moreover, 
      $$
      j(\rho(e\max x))=j(\rho(e))j(\rho(\max x))=e\max x.
      $$
      
      If $s\le t$ for some $s,t\in S$, then $s$ and $t$ belong to the same $\sigma$-class $x$ of $S$. Therefore, $s=e\max x$ and  $t=f\max x$ for $e,f\in E(S)$ with $e\le f$. So,
      $$
       \rho(s)=\rho(e)\rho(\max x)\le\rho(f)\rho(\max x)=\rho(t).
      $$
    \end{proof}
    
    \begin{rem}\label{rem-A->A->E(A)}
     The condition that $S$ is an $F$-inverse monoid in Proposition~\ref{prop-ext-by-F-inverse} is not necessary for the existence of an order-preserving transversal.  
    \end{rem}
    \noindent Indeed, the extension $A\overset{\id}{\to}A\overset{j}{\to}E(A)$, where $j(a)=aa\m$, admits the order-preserving transversal $\rho=\id_{E(A)}$. However, $E(A)$ is not an $F$-inverse monoid, if $A$ is not a monoid (nevertheless $E(A)$ is $E$-unitary).\\
    
    We shall need the following technical lemma.

\begin{lem}\label{lem-f(s',t')}
 Let $(\alpha,\lambda,f)$ be a Sieben twisted $S$-module structure on $A$. If $s\le s'$ and $t\le t'$, then
 \begin{align*}
  f(s,t)=\alpha(stt\m s\m)f(s',t').
 \end{align*}
\end{lem}
\begin{proof}
 Let $s=es'$, where $e\in E(S)$. Applying (v) of Definition~\ref{defn-twisted_S-mod} to the triple $(e,s',t)$, we get
 $$
  \lambda_e(f(s',t))f(e,s't)=f(e,s')f(s,t).
 $$
 Using (i) of the same definition, (iv') and the facts that $f(s',t)\in A_{\alpha(s'tt\m s'\m)}$ and $f(s,t)\in A_{\alpha(stt\m s\m)}$, we conclude that 
 \begin{align}\label{eq-f(s,t)=alpha(e)f(s',t)}
  f(s,t)=\alpha(e)f(s',t).
 \end{align}

 Now taking $e'\in E(S)$, such that $t=t'e'$, we apply (v) of Definition~\ref{defn-twisted_S-mod} to the triple $(s',t',e')$:
 $$
  \lambda_{s'}(f(t',e'))f(s',t)=f(s',t')f(s't',e').
 $$
 It follows that
 \begin{align}\label{eq-f(s',t)=alpha(s'tt^(-1)s'^(-1))f(s',t')}
  f(s',t)=\alpha(s'tt\m s'\m)f(s',t').
 \end{align}
 
 Combining~\eqref{eq-f(s,t)=alpha(e)f(s',t)} and~\eqref{eq-f(s',t)=alpha(s'tt^(-1)s'^(-1))f(s',t')}, one finally obtains
 \begin{align*}
  f(s,t)&=\alpha(es'tt\m s'\m)f(s',t')\\
  &=\alpha((es')tt\m(es')\m)f(s',t')\\
  &=\alpha(stt\m s\m)f(s',t').
 \end{align*}

\end{proof}

\begin{cor}\label{cor-alpha(st(s't')^(-1))f(s,t)}
 Let $(\alpha,\lambda,f)$ be a Sieben twisted $S$-module structure on $A,$ where $S$ is  $E$-unitary. If  $(s,s'),(t,t')\in\sigma ,$ then
 \begin{align*}
  \alpha(st(s't')\m)f(s,t)=\alpha(st(s't')\m)f(s',t').
 \end{align*}
\end{cor}
\noindent Indeed, since $S$ is $E$-unitary, then $s\m s'=s'\m s$ and $t\m t'=t'\m t$ are idempotents of $S$. So, $u=ss\m s'\le s,s'$ and $v=tt\m t'\le t,t'$. By Lemma~\ref{lem-f(s',t')}
$$
 f(u,v)=\alpha(uvv\m u\m)f(s,t)=\alpha(uvv\m u\m)f(s',t').
$$
It remains to prove that $uvv\m u\m=st(s't')\m$. We have
\begin{align*}
 uvv\m u\m&=ss\m \cdot s'\cdot tt\m\cdot t't'\m\cdot tt\m\cdot s'\m\cdot ss\m\\
 &=ss\m s't\cdot t\m t'\cdot t'\m s'\m\\
 &=ss\m s'\cdot tt'\m\cdot tt'\m\cdot s'\m\\
 &=s\cdot s\m s'\cdot tt'\m s'\m\\
 &=s\cdot s\m s'\cdot s\m s'\cdot tt'\m s'\m\\
 &=ss\m\cdot s's\m\cdot s'\cdot tt'\m s'\m\\
 &=s\cdot s\m s\cdot s'\m s'\cdot tt'\m\cdot s'\m\\
 &=s\cdot tt'\m\cdot s'\m\\
 &=st(s't')\m.
\end{align*}

\begin{cor}\label{cor-Lambda-for-A*_Lambda-S}
 If with the notations of Proposition~\ref{prop-rho-for-A*_Lambda-S} we assume that $\Lambda$ is a Sieben twisted $S$-module structure, then $\lambda'=\lambda$ and thus $\Lambda'=\Lambda$.
\end{cor}
\noindent For $f(s,\alpha\m(aa\m))=\alpha(s\alpha\m(aa\m)s\m)$ by (iv') and
$$
 f(s\alpha\m(aa\m),s\m)=\alpha(s\alpha\m(aa\m)s\m)f(s,s\m)
$$
by Lemma~\ref{lem-f(s',t')}. Therefore,~\eqref{eq-lambda'-for-A*_Lambda-S} becomes
\begin{align*}
 \lambda'_s(a)&=\lambda_s(a)f(s,s\m)\m\alpha(s\alpha\m(aa\m)s\m)f(s,s\m)\\
 &=\lambda_s(a)\alpha(ss\m)\alpha(s\alpha\m(aa\m)s\m)\\
 &=\lambda_s(a)\alpha(s\alpha\m(aa\m)s\m),
\end{align*}
the latter being $\lambda_s(a)$, as $a\in A_{ aa\m }$ and hence $\lambda_s(a)\in A_{\alpha(s\alpha\m(aa\m)s\m)}$.

    \section{\texorpdfstring{Extensions of $A$ by $G$}{Extensions of A by G}}\label{subsec-ext-A-G}
    
Let $A\overset{i}{\to}U\overset{j}{\to}S$ be an extension. Consider the composition $j'=\sigma^\natural\circ j:U\to \cG S$. When $S$ is $E$-unitary, $(\sigma^\natural)\m(1)=E(S)$. Therefore, $j'\m(1)=j\m(E(S))=i(A)$.
    
\begin{defn}\label{defn-ext-of-A-by-G}
 Let $A$ be a semilattice of groups and $G$ a group. An {\it extension of $A$ by $G$} is an inverse semigroup $U$ with a monomorphism $i:A\to U$ and an epimorphism $j:U\to G$, such that $i(A)=j\m(1)$.
\end{defn}

 Notice that $E(U)\subseteq j\m(1)=i(A)$, so\footnote{This will also follow from Proposition~\ref{prop-ext-A-G-to-ext-A-S}}
 \begin{align}\label{eq-E(U)=i(E(A))}
  E(U)=i(E(A)).
 \end{align}
 
\begin{rem}\label{rem-A-is-a-group}
 An extension $U$ of $A$ by $G$ is an extension of inverse semigroups in the sense of Lausch~\cite{Lausch} if and only if $A$ is a group. In this case $U$ is a classical extension of groups.
\end{rem}
\noindent For if $U$ is a ``Lausch'' extension, then $E(A)\cong E(U)\cong E(G)$, so $A$ and $U$ are groups. Conversely, if $A$ is a group, it follows from~\eqref{eq-E(U)=i(E(A))} that $U$ is also a group.

\begin{defn}\label{defn-equiv-ext-of-A-by-G}
Two extensions $A\overset{i}{\to}U\overset{j}{\to}G$ and $A\overset{i}{\to}U'\overset{j}{\to}G$ of $A$ by $G$ are called {\it equivalent} if there is an isomorphism $\mu:U\to U'$ making the following diagram
 	\begin{align}\label{eq-comm-diag-equiv-ext-A-G}
     	\begin{tikzpicture}[node distance=1.5cm, auto]
     		\node (A) {$A$};
     		\node (U) [right of=A] {$U$};
     		\node (G) [right of=U] {$G$};
     		\node (A') [below of=A]{$A$};
     		\node (U') [below of=U] {$U'$};
     		\node (G') [below of=S] {$G$};
     		\draw[->] (A) to node {$i$} (U);
     		\draw[->] (U) to node {$j$} (G);
     		\draw[->] (A') to node {$i'$} (U');
     		\draw[->] (U') to node {$j'$} (G');
     		\draw[-,double distance=2pt] (A) to node {} (A');
     		\draw[->] (U) to node {$\mu$} (U');
    		\draw[-,double distance=2pt] (G) to node {} (G');
     	\end{tikzpicture}
 	\end{align}
commute.
\end{defn}

\begin{rem}\label{rem-mu-is-injective}
 It is sufficient to require that $\mu$ is an epimorphism.
\end{rem}
 \noindent For the injectivity of $\mu$ follows from the commutativity of~\eqref{eq-comm-diag-equiv-ext-A-G}, the injectivity of $i$, $i'$ and the facts that $i(A)=j\m(1)$, $i'(A)=j'\m(1)$ (see the proof of Lemma~\ref{lem-mu-is-iso}).
 
\begin{prop}\label{prop-ext-A-G-to-ext-A-S} 
 Let $A\overset{i}{\to}U\overset{j}{\to}G$ be an extension of $A$ by $G$. Then there exist an inverse semigroup $S$ and epimorphisms $\pi:U\to S$, $\kappa:S\to G$, such that 
 \begin{enumerate}
  \item $A\overset{i}{\to}U\overset{\pi}{\to}S$ is an extension of $A$ by $S$;
  \item $j=\kappa\circ\pi$.
 \end{enumerate}
 Moreover, it follows that 
 \begin{enumerate}
  \item[(iii)] $S$ is $E$-unitary;
  \item[(iv)] $\ker\kappa=\sigma$.
 \end{enumerate}
\end{prop}
\begin{proof}
 Observe that $ui(A)u\m\subseteq i(A)$ for all $u\in U$, so $\mathcal A=\{i(A)_e\}_{e\in E(U)}$ is a (group) kernel normal system. Set $S=U/\rho_{\mathcal A}$ and $\pi=\rho_{\mathcal A}^\natural:U\to S$. By construction $\pi$ is an idempotent-separating epimorphism and $\pi\m(E(S))=i(A)$, giving (i). 
 
 To build $\kappa:S\to G$, note that $\ker\pi=\rho_{\mathcal A}\subseteq\ker j$. Indeed, $(s,t)\in\rho_{\mathcal A}$ means that $t=us$ for $u\in i(A)_{ss\m}$. Then $j(t)=j(us)=j(u)j(s)=j(s)$. Thus, there exists a (unique) epimorphism $\kappa:S\to G$ satisfying (ii). 
 
 To prove (iii), suppose that $e\le s$ for some $s=\pi(u)\in S$ and $e\in E(S)$. Then $\kappa(s)=\kappa(e)=1$, i.\,e. $j(u)=1$. Hence $u\in i(A)$ and consequently $s=\pi(u)\in E(S)$. 
 
 Clearly, $\ker\kappa$ is a group congruence, so $\sigma\subseteq\ker\kappa$. For the converse inclusion let $(s,t)\in\ker\kappa$, where $s=\pi(u)$, $t=\pi(v)$. Then $\kappa(s)=\kappa(t)$ means that $j(u)=j(v)$. It follows that $j(uv\m)=1$, so $uv\m\in i(A)$ and hence $st\m=\pi(uv\m)\in E(S)$. Similarly $s\m t\in E(S)$. Since $S$ is $E$-unitary, $(s,t)\in\sigma$ by~\cite[Theorem 2.4.6]{Lawson}. This proves (iv).
\end{proof}

\begin{prop}\label{prop-eq-ext-A-G-to-eq-ext-A-S}   
 Let $A\overset{i}{\to}U\overset{j}{\to}G$ and $A\overset{i'}{\to}U'\overset{j'}{\to}G$ be extensions of $A$ by $G$ with $\pi:U\to S$, $\kappa:S\to G$ and $\pi':U\to S'$, $\kappa':S'\to G$ being the corresponding epimorphisms from Proposition~\ref{prop-ext-A-G-to-ext-A-S}. Given a homomorphism $\mu:U\to U'$ making the diagram~\eqref{eq-comm-diag-equiv-ext-A-G} commute, there exists a homomorphism $\nu:S\to S'$, such that
 \begin{align}\label{eq-comm-diag-A-U-S-G}
     	\begin{tikzpicture}[node distance=1.5cm, auto]
     		\node (A) {$A$};
     		\node (U) [right of=A] {$U$};
     		\node (S) [right of=U] {$S$};
     		\node (G) [right of=S] {$G$};
     		\node (A') [below of=A]{$A$};
     		\node (U') [below of=U] {$U'$};
     		\node (S') [below of=S] {$S'$};
     		\node (G') [below of=G] {$G$};
     		\draw[->] (A) to node {$i$} (U);
     		\draw[->] (U) to node {$\pi$} (S);
     		\draw[->] (S) to node {$\kappa$} (G);
     		\draw[->] (A') to node {$i'$} (U');
     		\draw[->] (U') to node {$\pi'$} (S');
     		\draw[->] (S') to node {$\kappa'$} (G');
     		\draw[-,double distance=2pt] (A) to node {} (A');
     		\draw[->] (U) to node {$\mu$} (U');
     		\draw[->,dashed] (S) to node {$\nu$} (S');
     		\draw[-,double distance=2pt] (G) to node {} (G');
     	\end{tikzpicture}
	\end{align}
 	commutes. Moreover, if $\mu$ is injective, then $\nu$ is injective; if $\mu$ is surjective, then $\nu$ is surjective.
\end{prop}
\begin{proof}
 For the middle square of~\eqref{eq-comm-diag-A-U-S-G} to commute, we set $\nu(s)=\pi'\circ\mu(u)$, where $s=\pi(u)$. We need to show that the definition does not depend on the choice of $u$. If $\pi(u)=\pi(v)$, then $u=i(a)v$ with $i(aa\m)=vv\m$. Hence, $\mu(u)=i'(a)\mu(v)$ and consequently 
 \begin{align}\label{eq-pi'mu(u)=pi'i'(a)mu(v)}
  \pi'\circ\mu(u)=\pi'(i'(a))\pi'(\mu(v)).
 \end{align}
 Since $\pi'\circ i'(a)\in E(S')$, we have $\pi'\circ i'(a)=\pi'\circ i'(aa\m)=\pi'\circ\mu\circ i(aa\m)=\pi'\circ\mu(vv\m)$. Therefore,~\eqref{eq-pi'mu(u)=pi'i'(a)mu(v)} equals $\pi'\circ\mu(vv\m v)=\pi'\circ\mu(v)$ and thus $\nu$ is well-defined.

 To prove the commutativity of the right square of~\eqref{eq-comm-diag-A-U-S-G}, observe that $\kappa'\circ\nu(s)=\kappa'\circ\pi'\circ\mu(u)=j'\circ\mu(u)=j(u)=\kappa\circ\pi(u)=\kappa(s)$.
 
 Suppose that $\nu(s)=\nu(t)$ for $s=\pi(u)$ and $t=\pi(v)$. By the definition of $\nu$ we have $\pi'\circ\mu(u)=\pi'\circ\mu(v)$. Then $\mu(u)=i'(a)\mu(v)=\mu(i(a)v)$, where $\mu(i(aa\m))=i'(aa\m)=\mu(vv\m)$. If $\mu$ is injective, then $u=i(a)v$ with $i(aa\m)=vv\m$. It follows that $\pi(u)=\pi(v)$, i.\,e. $s=t$.
 
 Given $s'\in S'$, there is $u'\in U'$, such that $s'=\pi'(u')$. Assuming that $\mu$ is surjective, we find $u\in U$ with $\mu(u)=u'$. Setting $s=\pi(u)$, we see that $\nu(s)=\nu\circ\pi(u)=\pi'\circ\mu(u)=s'$.
\end{proof}

\begin{rem}\label{rem-S-pi-kappa-are-unique-up-to-iso}
 The semigroup $S$ from Proposition~\ref{prop-ext-A-G-to-ext-A-S} is unique up to an isomorphism respecting $\pi$ and $\kappa$.
\end{rem}
\noindent To see this, one sets $U=U'$ and $\mu=\id_U$ in Proposition~\ref{prop-eq-ext-A-G-to-eq-ext-A-S}.

\begin{prop}\label{prop-k^(-1)(x)-has-max}
 Under the conditions of Proposition~\ref{prop-ext-A-G-to-ext-A-S} the $\sigma$-class $\kappa\m(x)$ has a maximum element if and only if $j\m(x)=u_xi(A)$ for some  $u_x\in j\m(x)$.
\end{prop}
\begin{proof}
 Suppose that there is $s_x\in S$, such that $s\le s_x$ whenever $\kappa(s)=x$. Then choosing $u_x\in U$ with $\pi(u_x)=s_x$, for any $u\in U$ we have $j(u)=x$ if and only if $\kappa\circ\pi(u)=x$, i.\,e. $\pi(u)\le s_x=\pi(u_x)$. The latter is equivalent to $\pi(u)=\pi(u_xe)$, where $e\in E(U)$. Hence, $u=u_xei(a)$ for some $a\in A$ satisfying $i(a\m a)=eu_x\m u_x$. It remains to note that $ei(a)\in i(A)$, so $j\m(x)\subseteq u_xi(A)$. The converse inclusion is trivial.
 
 Assume that $j\m(x)=u_xi(A)$. Take $s\in S$ and observe that $\kappa(s)=x$ is equivalent to $j(u)=x$, where $\pi(u)=s$. So, $u=u_xi(a)$ and thus $s=\pi(u)\le\pi(u_x)$, because $\pi(i(a))\in E(S)$. Thus $\pi(u_x)$  is the maximum element of $\kappa\m(x)$.
 \end{proof}

\begin{cor}\label{cor-S-is-F-inv}
 Under the conditions of Proposition~\ref{prop-ext-A-G-to-ext-A-S} the semigroup $S$ is an $F$-inverse monoid if and only if there is $\{u_x\}_{x\in G}\subseteq U$, such that $j\m(x)=u_xi(A)$ for all $x\in G$.
\end{cor}

\section{Twisted partial actions}\label{subsec-tw-pact}

Let $S$ be a semigroup. By a {\it multiplier} of $S$ we mean a pair of linked right and left translations~\cite{Clifford-Preston-1,Petrich} of $S$, i.\,e. a pair of maps $(L,R)$ from $S$ to itself satisfying
\begin{enumerate}
 \item $L(st)=L(s)t$;
 \item $R(st)=sR(t)$;
 \item $sL(t)=R(s)t$
\end{enumerate}
for all $s,t\in S$. The multipliers of $S$ form a monoid\footnote{In the literature on semigroup theory the monoid $\mathcal M(S)$ is called {\it the translational hull of $S$}.} $\mathcal M(S)$ under the operation $(L,R)(L',R')=(L\circ L',R'\circ R)$. 

Given $w=(L,R)\in\mathcal M(S)$ and $s\in S$, we shall use the notations of~\cite{DES1} by writing $ws$ for $L(s)$ and $sw$ for $R(s)$. Then (i)--(iii) can be rewritten as 
\begin{enumerate}
 \item[(i')] $w(st)=(ws)t$;
 \item[(ii')] $(st)w=s(tw)$;
 \item[(iii')] $s(wt)=(sw)t$.
\end{enumerate}

\begin{rem}\label{rem-M(S)-cong-S}
 When $1\in S$ and $w\in\mathcal M(S)$, it follows from (iii') that $w1=1(w1)=(1w)1=1w$, and moreover the map $w\mapsto w1=1w$ defines an isomorphism of monoids $\mathcal M(S)\cong S$.
\end{rem}

\begin{rem}\label{rem-M(S)-for-comm-S}
 If $w\in\mathcal M(S)$ and $s,t\in C(S)$, then $w(st)=(ws)t=t(ws)=(tw)s=s(tw)=(st)w$. In particular, $we=ew$ for $e\in E(C(S))$.
\end{rem}

\begin{rem}\label{rem-wIw^(-1)}
 Let $I$ be an ideal of $S$ and $w\in\cM S$. If $I^2=I$, then $wI,Iw\subseteq I$. Moreover, if $w$ is invertible, then $wIw\m=I$.
\end{rem}
\noindent Indeed, as $I^2=I$, each $s\in I$ is $tu$ for some $t,u\in I$. Therefore, $ws=w(tu)=(wt)u\in I$ and similarly $sw=t(uw)\in I$. By (ii) of~\cite[Proposition 2.5]{DE} one has $(ws)w\m=w(sw\m)$ for all $s\in I$, so $wIw\m$ makes sense, being equal to $(wI)w\m=w(Iw\m)$. Clearly, $wIw\m\subseteq Iw\m\subseteq I$. For the converse inclusion observe that $w\m Iw\subseteq I$, so $I=w(w\m Iw)w\m\subseteq w Iw\m$.\\

We use $\cU M$ to denote the group of invertible elements of a monoid $M$.

\begin{lem}\label{lem-(ws)^(-1)}
 Let $S$ be an inverse semigroup and $w\in\cU{\cM S}$. Then for any $s\in S$:
 \begin{enumerate}
  \item $(ws)\m=s\m w\m$;
  \item $(sw)\m=w\m s\m$.
 \end{enumerate}
\end{lem}
\begin{proof}
 By (iii') of the definition of a multiplier $s\m w\m\cdot ws=((s\m w\m)w)s=s\m s$. Therefore, $ws\cdot s\m w\m\cdot ws=ws\cdot s\m s=w(ss\m s)=ws$. Similarly $s\m w\m\cdot ws\cdot s\m w\m=s\m s\cdot s\m w\m=(s\m ss\m)w\m=s\m w\m$, completing the proof of (i). Equality (ii) follows by replacing $s$ and $w$ by their inverses in (i).
\end{proof}

\begin{defn}\label{defn-tw_part_act}
 By a {\it twisted partial action}~\cite{DES1} of a group $G$ on a semigroup $S$ we mean a pair $\Theta=(\0,w)$, where $\0$ is a collection $\{\0_x:\cD_{x\m}\to \cD_x\}_{x\in G}$ of isomorphisms between non-empty ideals of $S$ and $w=\{w_{x,y}\in\cU{\mathcal M(\cD_x\cD_{xy})}\}_{x,y\in G}$, such that
 \begin{enumerate}
  \item $\cD_x^2=\cD_x$ and $\cD_x\cD_y=\cD_y\cD_x$;
  \item $\cD_1=S$ and $\0_1=\id_S$;
  \item $\0_x(\cD_{x\m}\cD_y)=\cD_x\cD_{xy}$;
  \item $\0_x\circ\0_y(s)=w_{x,y}\0_{xy}(s)w\m_{x,y}$ for any $s\in \cD_{y\m}\cD_{y\m x\m}$;
  \item $w_{1,x}=w_{x,1}=\id_{\cD_x}$;
  \item $\0_x(sw_{y,z})w_{x,yz}=\0_x(s)w_{x,y}w_{xy,z}$ for all $s\in \cD_{x\m}\cD_y\cD_{yz}$.
 \end{enumerate}
\end{defn}
Items (iv)--(vi) require some explanations. Since the ideals $\cD_x$ and $\cD_{xy}$ are idempotent and commute, their product $\cD_x\cD_{xy}$ is also an idempotent ideal, so by (ii) of~\cite[Proposition 2.5]{DE} we have $(ws)w'=w(sw')$ for all $w,w'\in\cM{\cD_x\cD_{xy}}$ and $s\in \cD_x\cD_{xy}$. This explains, why the right-hand side of (iv) makes sense without a pair of additional brackets. 

A priori $w_{1,x}\in\mathcal M(S\cdot \cD_x)$. However $S\cdot \cD_x=\cD_x$: the inclusion $S\cdot \cD_x\subseteq \cD_x$ is the fact that $\cD_x$ is a left ideal of $S$, while the converse inclusion takes place, because $\cD_x=\cD_x^2\subseteq S\cdot \cD_x$. So, the condition $w_{1,x}=\id_{\cD_x}$ in (v) makes sense.

To justify the applicability of the multipliers in (vi), we observe by (i) that (iii) extends to the product of any finite number of ideals $\cD_x$, namely
\begin{align*}
 \0_x(\cD_{x\m}\cD_{y_1}\dots \cD_{y_n})&=\0_x(\cD_{x\m}\cD_{y_1}\dots \cD_{x\m}\cD_{y_n})\\
 &=\cD_x\cD_{xy_1}\dots \cD_x\cD_{xy_n}\\
 &=\cD_x\cD_{xy_1}\dots \cD_{xy_n}.
\end{align*}

\begin{rem}\label{rem-unital-tw-part-act}
 If $\0$ is {\it unital} in the sense that all the ideals $\cD_x$, $x\in G$, are unital (that is $\cD_x=1_xS$ for some central idempotent $1_x\in S$), then $w$ can be identified with a function $G^2\ni(x,y)\mapsto w_{x,y}\in\mathcal U(1_x1_{xy}S)$ and (vi) reduces to
\begin{enumerate}
\item[(vi')] $\0_x(1_{x\m}w_{y,z})w_{x,yz}=w_{x,y}w_{xy,z}$.
\end{enumerate} 
Furthermore, in this case 
\begin{align}\label{D_x-cap-D_y=D_xD_y}
\cD _x \cap \cD _y= \cD _x  \cD _y,\    x,y\in G.
\end{align} 
If, moreover, $S$ is commutative, then  properties (i)--(iv) say that $\0$ is a partial action of $G$ on $S$, and (vi') is exactly the fact that $w$ is a partial $2$-cocycle (see~\cite{DK}).  
\end{rem}
\noindent Indeed, since the ideal $\cD_{y\m}\cD_{y\m x\m}$ is idempotent, then $\0_{xy}(s)=\0_{xy}(s's'')=\0_{xy}(s')\0_{xy}(s'')$ for some $s',s''\in \cD_{y\m}\cD_{y\m x\m}$. So, by Remark~\ref{rem-M(S)-for-comm-S} the right-hand side of (iv) becomes $\0_{xy}(s)w_{x,y}w\m_{x,y}=\0_{xy}(s)$.

\begin{rem}\label{I-cap-J=IJ-in-inverse-S} 
Notice that for any two ideals $I,J$ in an inverse semigroup $S$ one has that $I\cap J=I J,$ so~\eqref{D_x-cap-D_y=D_xD_y} holds for any twisted partial action on $S.$ Hence, any twisted partial action on a commutative inverse semigroup $S$ results in a partial action on $S$ (in the sense of~\cite{DK}).
\end{rem}

 
\begin{rem}\label{rem-part-act-on-E(S)}
 If $S$ is a semilattice of groups, then any twisted partial action of $G$ on $S$ restricts to a partial action of $G$ on $E(S)$.  
\end{rem}
\noindent Indeed, given a twisted partial action $(\0,w)$ of $G$ on $S$, we first observe that $E(\cD_x)$ is a non-empty ideal of $E(S)$, $E(\cD_1)=E(S)$ and $\0_x$ maps isomorphically $E(\cD_{x\m})$ onto $E(\cD_x)$. Furthermore, $E(IJ)=E(I)E(J)$ for any pair of ideals $I,J$ of $S$: the inclusion $E(I)E(J)\subseteq E(IJ)$ follows from the commutativity of $E(S)$, while the converse inclusion is explained by the fact that $E(IJ)\subseteq E(I)\cap E(J)$, so each $e\in E(IJ)$, being $e\cdot e$, belongs to $E(I)E(J)$. By this observation 
\begin{align*}
 \0_x(E(\cD_{x\m})E(\cD_y))=\0_x(E(\cD_{x\m}\cD_y))=E(\0_x(\cD_{x\m}\cD_y))&=E(\cD_x\cD_{xy})\\
 &=E(\cD_x)E(\cD_{xy}).
\end{align*}
Finally, $\0_x\circ\0_y(e)=\0_{xy}(e)$ for any $e\in E(\cD_{y\m})E(\cD_{y\m x\m})=E(\cD_{y\m}\cD_{y\m x\m})$ by Remark~\ref{rem-M(S)-for-comm-S}.

Following~\cite[Definition 6.1]{DES2} we give the next
\begin{defn}\label{defn-equiv-tw-pact}
 Two twisted partial actions $(\0,w)$ and $(\0',w')$ of $G$ on $S$ are called {\it equivalent}, if
 \begin{enumerate}
  \item $\cD'_x=\cD_x$;
  \item $\0'_x(s)=\e_x\0_x(s)\e\m_x$, $s\in\cD_{x\m}$;
  \item $\0'_x(s)w'_{x,y}\e_{xy}=\e_x\0_x(s\e_y)w_{x,y}$, $s\in\cD_{x\m}\cD_y$,
 \end{enumerate}
for some $\e_x\in\cU{\cM{\cD_x}}$.
\end{defn}

\begin{rem}\label{rem-eq-tw-pact-rest-to-E(A)}
 Let $A$ be a semilattice of groups. Then equivalent twisted partial actions of $G$ on $A$ restrict to the same partial action of $G$ on $E(A)$.
\end{rem}
\noindent Indeed, if $(\0,w)$ is equivalent to $(\0',w')$, then $E(\cD_x)=E(\cD'_x)$ due to (i) of Definition~\ref{defn-equiv-tw-pact}. Moreover, $\0'_x(e)=\e_x\0_x(e)\e\m_x=\0_x(e)$ for any $e\in E(\cD_{x\m})$ by (ii) of the same definition and Remark~\ref{rem-M(S)-for-comm-S}.

\begin{lem}\label{lem-0_x-inv}
 Let $(\0,w)$ be a twisted partial action of a group $G$ on a semigroup $S$. Then for any $x\in G$ and $s\in \cD_x$ we have\footnote{This is formula (14) from~\cite{DES1}}
 \begin{align}\label{eq-0_x-inv}
  \0\m_x(s)=w\m_{x\m,x}\0_{x\m}(s)w_{x\m,x}.
 \end{align}
\end{lem}
\begin{proof}
 Denote the right-hand side of~\eqref{eq-0_x-inv} by $f_x(s)$. Obviously, $f_x:\cD_x\to \cD_{x\m}$. By (ii) and (iv) of Definition~\ref{defn-tw_part_act} the composition $f_x\circ\0_x(s)$ is
 $$
  w\m_{x\m,x}\0_{x\m}(\0_x(s))w_{x\m,x}=w\m_{x\m,x} w_{x\m,x}sw\m_{x\m,x} w_{x\m,x}=s.
 $$
 
 Let us now prove that $\0_x\circ f_x(s)=s$ for $ s\in \cD _x.$  Writing (vi) of Definition~\ref{defn-tw_part_act} with $y=x\m$, $z=x$ and taking into account (v), we get\footnote{This is formula (6) from~\cite{DES1}}
 \begin{align}\label{eq-0_x(sw_(x^(-1)x))}
  \0_x(sw_{x\m x})=\0_x(s)w_{x,x\m}
 \end{align}
 for all $x\in G$ and $s\in \cD_{x\m}$. Then $\0_x\circ f_x(s)$ is
 \begin{align}\label{eq-0_x-circ-f_x(s)}
  \0_x(w\m_{x\m,x}\0_{x\m}(s)w_{x\m,x})=\0_x(w\m_{x\m,x}\0_{x\m}(s))w_{x,x\m}.
 \end{align}
 By Lemma~\ref{lem-(ws)^(-1)}
 $$
  w\m_{x\m,x}\0_{x\m}(s)=(\0_{x\m}(s)\m w_{x\m,x})\m=(\0_{x\m}(s\m)w_{x\m,x})\m.
 $$
 It follows, using~\eqref{eq-0_x(sw_(x^(-1)x))}, that the right-hand side of~\eqref{eq-0_x-circ-f_x(s)} equals
 $$
   \0_x(\0_{x\m}(s\m)w_{x\m,x})\m w_{x,x\m}=(\0_x(\0_{x\m}(s\m))w_{x,x\m})\m w_{x,x\m},
 $$
 which is 
 $$
  (w_{x,x\m}s\m w\m_{x,x\m} w_{x,x\m})\m w_{x,x\m}=(w_{x,x\m}s\m)\m w_{x,x\m}
 $$ 
 thanks to (ii) and (iv) of Definition~\ref{defn-tw_part_act}. It remains to apply Lemma~\ref{lem-(ws)^(-1)} to get $sw\m_{x,x\m} w_{x,x\m}=s$.
\end{proof}

Having a twisted partial action $\Theta=(\theta,w)$ of $G$ on $S$, one can define the {\it crossed product of $S$ and $G$ by $\Theta$} as the set $S*_\Theta G=\{s\delta_x\mid s\in \cD_x\}$ with the multiplication $s\delta_x\cdot t\delta_y=\0_x(\0\m_x(s)t)w_{x,y}\delta_{xy}$. The proof of Theorem~2.4 from~\cite{DES1} implies that $S*_\Theta G$ is associative, so $S*_\Theta G$ is a semigroup.

\begin{lem}\label{lem-S*_Theta_G_is_inverse}
 Let $\Theta=(\theta,w)$ be a twisted partial action of a group $G$ on a semigroup $S$. If $S$ is inverse, then $S*_\Theta G$ is inverse and 
\begin{align}\label{eq-s-delta_x-inv}
 (s\delta_x)\m=w\m_{x\m,x}\0_{x\m}(s\m)\delta_{x\m}
\end{align}
 in $S*_\Theta G$.
\end{lem}
\begin{proof}
 It is easy to see that the idempotents of $S*_\Theta G$ are precisely the elements of the form $e\delta_1$ with $e\in E(S)$. Obviously, any two such elements commute.

Denote the right-hand side of~\eqref{eq-s-delta_x-inv} by $t\delta_{x\m}$ and calculate the product $t\delta_{x\m}\cdot s\delta_x$:
\begin{align}\label{eq-t-delta_(x^(-1))-cdot-s-delta_x}
 t\delta_{x\m}\cdot s\delta_x=\0_{x\m}(\0\m_{x\m}(t)s)w_{x\m,x}\delta_1.
\end{align}
By Lemmas~\ref{lem-0_x-inv} and~\ref{lem-(ws)^(-1)}
\begin{align*}
 \0\m_{x\m}(t)&=w\m_{x,x\m}\0_x(t)w_{x,x\m}\\
 &=w\m_{x,x\m}\0_x(w\m_{x\m,x}\0_{x\m}(s\m))w_{x,x\m}\\
 &=w\m_{x,x\m}\0_x(\0_{x\m}(s)w_{x\m,x})\m w_{x,x\m}.
\end{align*}
Thanks to~\eqref{eq-0_x(sw_(x^(-1)x))} and (ii), (iv) of Definition~\ref{defn-tw_part_act} this equals
\begin{align*}
 w\m_{x,x\m}(\0_x(\0_{x\m}(s))w_{x,x\m})\m w_{x,x\m}&=w\m_{x,x\m}(w_{x,x\m}s)\m w_{x,x\m}\\
 &=w\m_{x,x\m} s\m.
\end{align*}
Hence, the right-hand side of~\eqref{eq-t-delta_(x^(-1))-cdot-s-delta_x} is
$$
 \0_{x\m}(w\m_{x,x\m} s\m s)w_{x\m,x}\delta_1=\0_{x\m}(s\m s w_{x,x\m})\m w_{x\m,x}\delta_1.
$$
Using~\eqref{eq-0_x(sw_(x^(-1)x))} and Lemma~\ref{lem-0_x-inv} we finally see that 
\begin{align}\label{eq-0_x^(-1)(s^(-1)s)-delta_1}
 t\delta_{x\m}\cdot s\delta_x&=(\0_{x\m}(s\m s) w_{x\m,x})\m w_{x\m,x}\delta_1\notag\\
 &=w\m_{x\m,x}\0_{x\m}(s\m s)w_{x\m,x}\delta_1\notag\\
 &=\0\m_x(s\m s)\delta_1.
\end{align}

It immediately follows from~\eqref{eq-0_x^(-1)(s^(-1)s)-delta_1} and (v) of Definition~\ref{defn-tw_part_act} that
\begin{align*}
 s\delta_x\cdot t\delta_{x\m}\cdot s\delta_x&=s\delta_x\cdot\0\m_x(s\m s)\delta_1\\
 &=\0_x(\0\m_x(s)\0\m_x(s\m s))\delta_x\\
 &=\0_x(\0\m_x(ss\m s))\delta_x\\
 &=s\delta_x.
\end{align*}
 
At the same time $t\delta_{x\m}\cdot s\delta_x\cdot t\delta_{x\m}$ by~\eqref{eq-0_x^(-1)(s^(-1)s)-delta_1} and (ii), (v) of Definition~\ref{defn-tw_part_act} equals
$$
 \0\m_x(s\m s)\delta_1\cdot w\m_{x\m,x}\0_{x\m}(s\m)\delta_{x\m}=\0\m_x(s\m s)w\m_{x\m,x}\0_{x\m}(s\m)\delta_{x\m}.
$$
Applying Lemma~\ref{lem-0_x-inv} twice, we conclude
\begin{align*}
 t\delta_{x\m}\cdot s\delta_x\cdot t\delta_{x\m}&=\0\m_x(s\m s)\0\m_x(s\m)w\m_{x\m,x}\delta_{x\m}\\
 &=\0\m_x(s\m)w\m_{x\m,x}\delta_{x\m}\\
 &=w\m_{x\m,x}\0_{x\m}(s\m)\delta_{x\m}\\
 &=t\delta_{x\m}.
\end{align*}
	\end{proof}

\begin{cor}\label{cor-sdelta_x(sdelta_x)^(-1)}
 For any $s\delta_x\in A*_\Theta G$ one has $s\delta_x(s\delta_x)\m=ss\m\delta_1$ and $(s\delta_x)\m s\delta_x=\0\m_x(s\m s)\delta_1$.
\end{cor}
\noindent Indeed, the second equality is~\eqref{eq-0_x^(-1)(s^(-1)s)-delta_1}. Then using~\eqref{eq-0_x-inv}, \eqref{eq-s-delta_x-inv} and Lemma~\ref{lem-(ws)^(-1)} we have
\begin{align*}
 s\delta_x(s\delta_x)\m&=((s\delta_x)\m)\m(s\delta_x)\m\\
 &=(w\m_{x\m,x}\0_{x\m}(s\m)\delta_{x\m})\m w\m_{x\m,x}\0_{x\m}(s\m)\delta_{x\m}\\
 &=\0\m_{x\m}((w\m_{x\m,x}\0_{x\m}(s\m))\m w\m_{x\m,x}\0_{x\m}(s\m))\delta_1\\
 &=\0\m_{x\m}(\0_{x\m}(ss\m))\delta_1\\
 &=w\m_{x,x\m}\0_x(\0_{x\m}(ss\m))w_{x,x\m}\delta_1\\
 &=w\m_{x,x\m}w_{x,x\m}ss\m w\m_{x,x\m}w_{x,x\m} \delta_1\\
 &=ss\m\delta_1.
\end{align*} 
 
\begin{lem}\label{lem-S*_Theta-G-cong-S*_Theta'-G}
 Let $\Theta=(\0,w)$ and $\Theta'=(\0',w')$ be twisted partial actions of $G$ on $S$, such that the map $\varphi(s\delta_x)=s\delta'_x$ is a well-defined isomorphism $S*_\Theta G\to S*_{\Theta'}G$. Then $\Theta=\Theta'$.
\end{lem}
\begin{proof}
 We have 
 $$
  s\in \cD_x\iff s\delta_x\in S*_\Theta G\iff\varphi(s\delta_x)=s\delta'_x\in S*_{\Theta'}G\iff s\in \cD'_x,
 $$
 so $\0$ and $\0'$ have the same domains. Since
 $$
  ss\m\delta_x\cdot\0\m_x(s)\delta_1=\0_x(\0\m_x(ss\m)\0\m_x(s))\delta_x=\0_x\circ\0\m_x(ss\m s)\delta_x=s\delta_x,
 $$
 one has $\varphi(s\delta_x)=\varphi(ss\m\delta_x)\varphi(\0\m_x(s)\delta_x)$, that is
 $$
  s\delta'_x=\0'_x(\0'\m_x(ss\m)\0\m_x(s))\delta'_x=ss\m\0'_x(\0\m_x(s))\delta'_x.
 $$
 Here we used the facts that $\0'\m_x(ss\m)\in \cD'_{x\m}$, $\0\m_x(s)\in \cD_{x\m}$ and $\cD'_{x\m}=\cD_{x\m}$. It follows  that 
 \begin{align}\label{eq-s<=0'-circ-0_x(s)}
  s\le\0'_x(\0\m_x(s)).
 \end{align}
 By symmetry $s\le\0_x(\0'\m_x(s))$. But applying the isomorphism $\0_x\circ\0'\m_x:\cD_x\to \cD_x$ to the both sides of~\eqref{eq-s<=0'-circ-0_x(s)}, we get $\0_x(\0'\m_x(s))\le s$. Thus, $\0_x(\0'\m_x(s))=s$, that is $\0_x=\0'_x$.
 
 Let $s\in \cD_{x\m}$ and $t\in \cD_y$. Then
 $$
  \0_x(s)\delta_x\cdot t\delta_y=\0_x(\0\m_x(\0_x(s))t)w_{x,y}\delta_{xy}=\0_x(st)w_{x,y}\delta_{xy}.
 $$
 Since
 $$
  \varphi(\0_x(s)\delta_x)\varphi(t\delta_y)=\0_x(st)w'_{x,y}\delta'_{xy},
 $$
 it follows that $\0_x(st)w_{x,y}=\0 _x(st)w'_{x,y}$. By (iii) of Definition~\ref{defn-tw_part_act} each element of $\cD_x\cD_{xy}$ has the  form $\0_x(st)$ for some $s\in \cD_{x\m}$ and $t\in \cD_y$. Thus, the ``right parts'' of the multipliers $w_{x,y}$ and $w'_{x,y}$ coincide. Taking into account Lemma~\ref{lem-(ws)^(-1)}, we conclude that $w_{x,y}=w'_{x,y}$.
\end{proof}

\begin{prop}\label{prop-A*_Theta-G-is-ext}
 Let $G$ be a group and $A$ a semilattice of groups. For any twisted partial action $\Theta=(\0,w)$ of $G$ on $A$ the crossed product $A*_\Theta G$ is an extension of $A$ by $G$. Moreover, up to an isomorphism, the corresponding $E$-unitary inverse semigroup is $E(A)*_\0 G$ with $\pi:A*_\Theta G\to E(A)*_\0 G$ and $\kappa:E(A)*_\0 G\to G$ given by $\pi(a\delta_x)=aa\m\delta_x$ and $\kappa(e\delta_x)=x$.
\end{prop}
\begin{proof}
 By Lemma~\ref{lem-S*_Theta_G_is_inverse} the semigroup $A*_\Theta G$ is inverse. The map $i:A\to A*_\Theta G$, $i(a)=a\delta_1$, is well-defined and clearly injective. Thanks to (ii) and (v) of Definition~\ref{defn-tw_part_act} it is a homomorphism. An epimorphism $j:A*_\Theta G\to G$ can be naturally defined by $j(a\delta_x)=x$. Obviously, $j\m(1)=\{a\delta_1\mid a\in A\}=i(A)$.
 
 In view of Lemma~\ref{lem-(ws)^(-1)}, (iii') of the definition of a multiplier and the fact that $E(A)\subseteq C(A)$ we have
 \begin{align*}
  \pi(a\delta_x\cdot b\delta_y)&=\0_x(\0\m_x(a)b)w_{x,y}\cdot w\m_{x,y}\0_x(b\m\0\m_x(a\m))\delta_{xy}\\
  &=\0_x(\0\m_x(a)b)\0_x(b\m\0\m_x(a\m))\delta_{xy}\\
  &=\0_x(\0\m_x(a)bb\m\0\m_x(a\m))\delta_{xy}\\
  &=\0_x(\0\m_x(aa\m)bb\m)\delta_{xy}\\
  &=aa\m\delta_x\cdot bb\m\delta_y\\
  &=\pi(a\delta_x)\pi(b\delta_y).
 \end{align*}
It is clear that $\pi$ is surjective. Since $E(A*_\Theta G)=E(E(A)*_\0 G)=i(E(A))$, it follows that $\pi|_{E(A*_\Theta G)}=\id$ (so $\pi$ is idempotent-separating) and $\pi\m(E(E(A)*_\0 G))=\pi\m\circ i(E(A))=i(A)$. Moreover,
$$
 \kappa\circ\pi(a\delta_x)=\kappa(aa\m\delta_x)=x=j(a\delta_x).
$$
Thus, (i)--(ii) of Proposition~\ref{prop-ext-A-G-to-ext-A-S} are satisfied, and therefore (iii)--(iv) of the same proposition also hold.
\end{proof}

Observe that the extension $A\overset{i}{\to}A*_\Theta G\overset{\pi}{\to}E(A)*_\0 G$ from Proposition~\ref{prop-A*_Theta-G-is-ext} admits the order-preserving transversal $\rho=\id_{E(A)*_\0 G}$.

\begin{defn}\label{defn-adm-ext-of-A-by-G}
 Let $A$ be a semilattice of groups and $G$ a group. An extension $A\overset{i}{\to}U\overset{j}{\to}G$ will be called {\it admissible}, if there exists an inverse semigroup $S$ together with an epimorphism $\pi:U\to S$ as in Proposition~\ref{prop-ext-A-G-to-ext-A-S}, such that $\pi$ has an order-preserving transversal.
\end{defn}
 
\begin{rem}\label{rem-adm-is-inv}
 Admissibility is invariant under equivalence of extensions of $A$ by $G$.
\end{rem}
\noindent This follows from Remark~\ref{rem-equiv-ext-Sieben} and Proposition~\ref{prop-eq-ext-A-G-to-eq-ext-A-S}.

Our next aim is to show that any admissible extension of $A$ by $G$ is equivalent to $A*_\Theta G$ for some twisted partial action $\Theta$ of $G$ on $A$.

\section{From Sieben twisted modules to twisted partial actions}\label{sec-Theta^Lambda}

\begin{lem}\label{lem-ss^(-1)=tt^(-1)}
 Let $S$ be an inverse semigroup and $s,t\in S$. If $s\m t\in E(S)$ and $ss\m=tt\m$, then $s=t$.
\end{lem}
\begin{proof}
 We have
 \begin{align*}
  s&=ss\m\cdot s=tt\m\cdot s=t(s\m t)\m=t\cdot s\m t\le t,\\
  t&=tt\m\cdot t=ss\m\cdot t=s\cdot s\m t\le s,
 \end{align*}
 whence $s=t$.
\end{proof}

\begin{cor}\label{cor-s=t-in-E-unitary}
 If $S$ is $E$-unitary and $s,t\in S$ with $(s,t)\in\sigma$, then $ss\m=tt\m$ implies $s=t$.
\end{cor}
\noindent Indeed, $(s,t)\in\sigma$ means that $s\m t,st\m\in E(S)$.

\begin{rem}\label{rem-s^(-1)s=t^(-1)t}
 Corollary~\ref{cor-s=t-in-E-unitary} remains valid, if we replace $ss\m=tt\m$ by $s\m s=t\m t$.
\end{rem}
\noindent For in this case we may apply Lemma~\ref{lem-ss^(-1)=tt^(-1)} to the pair $s\m,t\m\in S$.\\

In all what follows in this section  $S$ is assumed to be an $E$-unitary inverse semigroup and $A$ a semilattice of groups with a twisted $S$-module structure $\Lambda=(\alpha,\lambda,f)$.

\begin{lem}\label{lem-D_x-ideal}
 For each $\sigma$-class $x\in\cG S$ set 
\begin{align}\label{eq-D_x-from-Lambda}
 \cD_x=\bigsqcup_{s\in x} A_{\alpha(ss\m)}.
\end{align}
Then $\{\cD_x\}_{x\in\cG S}$ is a collection of commuting idempotent ideals of $A$ with $\cD_1=A$.
\end{lem}
\begin{proof}
 First of all we note that the union above is indeed disjoint, because by Corollary~\ref{cor-s=t-in-E-unitary} for distinct $s,t\in x$ the idempotents $\alpha(ss\m)$ and $\alpha(tt\m)$ should also be distinct.
 
 Suppose that $a\in \cD_x$, that is $a\in A_{\alpha(ss\m)}$ for some $s\in x$. Given an arbitrary $b\in A$, consider $t=\alpha\m(bb\m)s$. Clearly $t\in x$, as $t\le s$. Moreover,
 $$
  \alpha(tt\m)=\alpha(\alpha\m(bb\m)ss\m)=bb\m\alpha(ss\m)=bb\m aa\m,
 $$
 the latter being $baa\m b\m=ba(ba)\m$, because $E(A)\subseteq C(A)$. Thus, $ba\in \cD_x$. The fact that $ab\in \cD_x$ is proved similarly.
 
 The inclusion $\cD_x^2\subseteq \cD_x$ is trivial. Now given $a\in \cD_x$, we have $a=aa\m\cdot a$ with $aa\m\in \cD_x$, so $a\in \cD_x^2$.
 
 Take $a\in \cD_x$, $b\in \cD_y$ and find $s\in x$, $t\in y$ with $a\in A_{\alpha(ss\m)}$, $b\in A_{\alpha(tt\m)}$. Note that
 $$
  ab=a\cdot a\m a\cdot bb\m\cdot b=abb\m\cdot a\m ab.
 $$
 We prove that $abb\m\in \cD_y$ and $a\m ab\in \cD_x$. Indeed, 
 $$
  abb\m\in A_{\alpha(ss\m)}A_{\alpha(tt\m)}\subseteq A_{\alpha(ss\m tt\m)}=A_{\alpha(ss\m t(ss\m t)\m)}
 $$
 and $ss\m t\in y$, as $ss\m t\le t$. The proof that $a\m a b=baa\m\in \cD_x$ is symmetric.
 
 Since $S$ is $E$-unitary, then $s\in 1$ is equivalent to $s\in E(S)$. So, $\cD_1=\bigsqcup_{e\in E(S)} A_{\alpha(e)}=\bigsqcup_{e\in E(A)} A_e=A$.
\end{proof}

 Now we assume additionally that $\Lambda$ is a Sieben twisted $S$-module structure.
 
\begin{lem}\label{lem-0_x-isom}
 For any $a\in \cD_{x\m}$ set 
 \begin{align}\label{eq-0_x-from-Lambda}
  \0_x(a)=\lambda_s(a),\mbox{ where }s\in x\mbox{ and }\alpha(s\m s)=aa\m. 
 \end{align}
 Then $\0_x$ is a well-defined isomorphism $\cD_{x\m}\to \cD_x$ satisfying (ii) of Definition~\ref{defn-tw_part_act}.
\end{lem}
\begin{proof}
 By definition $\cD_{x\m}=\bigsqcup_{s\in x\m} A_{\alpha(ss\m)}=\bigsqcup_{s\in x} A_{\alpha(s\m s)}$, so $\0_x(a)$ is well-defined.
 
 Since $\cD_{x\m}$ and $\cD_x$ are disjoint unions of the blocks $A_{\alpha(s\m s)}$ and $A_{\alpha(ss\m)}$, respectively ($s\in x$), and $\0_x|_{A_{\alpha(s\m s)}}=\lambda_s|_{A_{\alpha(s\m s)}}$ is an isomorphism $A_{\alpha(s\m s)}\to A_{\alpha(ss\m)}$ thanks to Remarks~\ref{lambda-iso} and~\ref{lambda-bar},  then $\0_x$ is a bijection $\cD_{x\m}\to \cD_x$. We shall prove that $\0_x$ is a homomorphism. 
 
 Given $a,b\in \cD_{x\m}$, we find a (unique) pair of $s,t\in x$, such that $a\in A_{\alpha(s\m s)}$ and $b\in A_{\alpha(t\m t)}$. Then $\0_x(a)\0_x(b)$ is by definition $\lambda_s(a)\lambda_t(b)$. To calculate $\0_x(ab)$, we observe that 
 $$
 ab\in A_{\alpha(s\m s)}A_{\alpha(t\m t)}\subseteq A_{\alpha(s\m st\m t)}=A_{\alpha((ts\m s)\m ts\m s)} 
 $$
 and $ts\m s\in x$, so $\0_x(ab)=\lambda_{ts\m s}(ab)$. Using (iii) of Definition~\ref{defn-twisted_S-mod}, we rewrite this as
 \begin{align}
  \lambda_{ts\m s}(ab)&=\lambda_{ts\m s}(a)\lambda_{ts\m s}(b)\notag\\
  &=(\xi_{f(ts\m,s)\m}\circ\lambda_{ts\m}\circ\lambda_s(a))(\xi_{f(t,s\m s)\m}\circ\lambda_t\circ\lambda_{s\m s}(b))\label{eq-lambda_(ts^(-1)s)}.
 \end{align}
 Since $(s,t)\in\sigma$ and $S$ is $E$-unitary, then $ts\m=st\m$ and $t\m s=s\m t$ are idempotents. This means first that $\lambda_{ts\m}$ is the multiplication by $\alpha(st\m)$ thanks to (i) of Definition~\ref{defn-twisted_S-mod}. Moreover,
 $$
  f(ts\m,s)=f(t,s\m s)=\alpha(ts\m st\m)=\alpha(st\m),
 $$
 in view of (iv'). Thus,~\eqref{eq-lambda_(ts^(-1)s)} reduces to
 \begin{align*}
 \alpha(st\m)\lambda_s(a)\lambda_t(\alpha(s\m s)b)&=\alpha(st\m)\lambda_s(a)\alpha(ts\m st\m)\lambda_t(b)\\
  &=\alpha(st\m)\lambda_s(a)\lambda_t(b).
 \end{align*}
 Here we applied (ii) of Definition~\ref{defn-twisted_S-mod}. It remains to notice that 
 $$
  \lambda_s(a)\lambda_t(b)\in A_{\alpha(ss\m)}A_{\alpha(tt\m)}\subseteq A_{\alpha(ss\m tt\m)}=A_{\alpha(s\cdot s\m t\cdot t\m)}=A_{\alpha(st\m)}.
  $$

  Thus, $\0_x$ is an isomorphism $\cD_{x\m}\to \cD_x$ with 
 $$
  \0_1(a)=\lambda_{\alpha\m(aa\m)}(a)=\alpha(\alpha\m(aa\m))a=aa\m\cdot a=a
 $$ 
 for all $a\in \cD_1=A$.
 \end{proof}
 
 \begin{lem}\label{lem-0_x(D_(x^(-1))D_y)}
  For all $x,y\in\cG S$ one has that
  \begin{align*}
  \0_x(\cD_{x\m}\cD_y)\subseteq \cD_x\cD_{xy}. 
 \end{align*}
 \end{lem}
 \begin{proof}
 Let $a=bc\in \cD_{x\m}\cD_y$, where $b\in \cD_{x\m}$ and $c\in \cD_y$. Choose $s\in x$ and $t\in y$, such that $b\in A_{\alpha(s\m s)}$ and $c\in A_{\alpha(tt\m)}$. Then
 $$
  a\in A_{\alpha(s\m s)}A_{\alpha(tt\m)}\subseteq A_{\alpha(s\m stt\m)}=A_{\alpha((stt\m)\m stt\m)}
 $$
 and consequently $\0_x(a)=\lambda_{stt\m}(bc)$. Notice that conjugating $s\m s$ and $tt\m$ by $stt\m$, we get the same idempotent $stt\m s\m$ of $S$. Therefore, both $\lambda_{stt\m}(b)$ and $\lambda_{stt\m}(c)$ belong to $A_{\alpha(stt\m s\m)}$.
 Since 
 $$
  stt\m s\m=stt\m(stt\m)\m=st(st)\m
 $$ 
 with $stt\m\in x$ and $st\in xy$, it follows that $A_{\alpha(stt\m s\m)}\subseteq \cD_x\cap \cD_{xy}$. Thus,
 $$
  \0_x(a)=\lambda_{stt\m}(b)\lambda_{stt\m}(c)\in \cD_x\cD_{xy}.
 $$
\end{proof}

\begin{lem}\label{lem-w_(x,y)-in-M(D_xD_xy)}
 Given $a\in \cD_x\cD_{xy}$, define 
 \begin{align}\label{eq-w_(x,y)-from-Lambda}
  w_{x,y}a=f(s,s\m t)a,\ \ aw_{x,y}=af(s,s\m t),
 \end{align}
  where $s\in x$ and $t\in xy$, such that $\alpha(ss\m)=\alpha(tt\m)=aa\m$.
 Then $w_{x,y}$ is a well-defined invertible multiplier of $\cD_x\cD_{xy}$.
\end{lem}
\begin{proof}
 To show that the definition of $w_{x,y}$ makes sense, observe that 
 $$
  \cD_x\cD_{xy}\subseteq \cD_x\cap \cD_{xy}=\left(\bigsqcup_{s\in x}A_{\alpha(ss\m)}\right)\cap\left(\bigsqcup_{t\in xy}A_{\alpha(tt\m)}\right),
 $$
 the latter being the union of those $A_{\alpha(e)}$, for which there exists a (unique) pair of $s\in x$ and $t\in xy$, such that $ss\m=tt\m=e$. Moreover, $w_{x,y}a$ and $aw_{x,y}$, thus defined, belong to $\cD_x\cD_{xy}$, as they are invertible with respect to
 $$
  aa\m f(s,s\m t)f(s,s\m t)\m=aa\m\alpha(ss\m tt\m)=\alpha(ss\m)=\alpha(tt\m).
 $$
 
 Let $a,b\in \cD_x\cD_{xy}$. We shall prove that $w_{x,y}(ab)=(w_{x,y}a)b$. Choose $s,u\in x$ and $t,v\in xy$ with
 \begin{align}
  aa\m&=\alpha(ss\m)=\alpha(tt\m),\label{eq-aa^(-1)=alpha(ss^(-1))=alpha(tt^(-1))}\\ 
  bb\m&=\alpha(uu\m)=\alpha(vv\m).\label{eq-bb^(-1)=alpha(uu^(-1))=alpha(vv^(-1))}
 \end{align}
 Notice that $ab(ab)\m=aa\m bb\m$ on the one hand equals
 $$
  \alpha(ss\m uu\m)=\alpha(ss\m u(ss\m u)\m),
 $$
 and on the other hand it is
 $$
  \alpha(tt\m vv\m)=\alpha(tt\m v(tt\m v)\m).
 $$
 Since $ss\m u\in x$ and $tt\m v\in xy$, it follows that
 \begin{align}\label{eq-w_(x,y)ab}
  w_{x,y}(ab)=f(ss\m u,(ss\m u)\m tt\m v)ab=f(s\cdot s\m u,u\m s\cdot s\m t\cdot t\m v)ab.
 \end{align}
 As $S$ is $E$-unitary and $(s,u),(t,v)\in\sigma$, we have that $s\m u=u\m s\in E(S)$ and $t\m v\in E(S)$. Hence $s\cdot s\m u\le s$ and $u\m s\cdot s\m t\cdot t\m v\le s\m t$. By Lemma~\ref{lem-f(s',t')}
 \begin{align*}
  f(s\cdot s\m u,u\m s\cdot s\m t\cdot t\m v)=\alpha(e)f(s,s\m t),
 \end{align*}
 where 
 \begin{align}\label{eq-e=ss^(-1)uu^(-1)}
  e=ss\m uu\m tt\m vv\m=ss\m uu\m
 \end{align}
 thanks to~\eqref{eq-aa^(-1)=alpha(ss^(-1))=alpha(tt^(-1))}--\eqref{eq-bb^(-1)=alpha(uu^(-1))=alpha(vv^(-1))}. Thus, one concludes from~\eqref{eq-w_(x,y)ab}--\eqref{eq-e=ss^(-1)uu^(-1)} that
 $$
  w_{x,y}(ab)=\alpha(ss\m uu\m)f(s,s\m t)ab,
 $$
 the latter being
 $$
  aa\m bb\m f(s,s\m t)ab=f(s,s\m t)ab=(w_{x,y}a)b
 $$
 in view of~\eqref{eq-aa^(-1)=alpha(ss^(-1))=alpha(tt^(-1))}--\eqref{eq-bb^(-1)=alpha(uu^(-1))=alpha(vv^(-1))}.
 
 The proof that $(ab)w_{x,y}=a(bw_{x,y})$ is symmetric.
 
 The property $(aw_{x,y})b=a(w_{x,y}b)$ is equivalent to 
 \begin{align}\label{eq-af(s,s^(-1)t)b=af(u,u^(-1)v)b}
  af(s,s\m t)b=af(u,u\m v)b.
 \end{align}
 Since $(s,u),(s\m t,u\m v)\in\sigma$, it follows from Corollary~\ref{cor-alpha(st(s't')^(-1))f(s,t)} that 
 \begin{align}\label{eq-alpha(ss^(-1)t(uu^(-1)v)^(-1))f(s,s^(-1)t)}
  \alpha(ss\m t(uu\m v)\m)f(s,s\m t)=\alpha(ss\m t(uu\m v)\m)f(u,u\m v).
 \end{align}
 Using~\eqref{eq-aa^(-1)=alpha(ss^(-1))=alpha(tt^(-1))}--\eqref{eq-bb^(-1)=alpha(uu^(-1))=alpha(vv^(-1))}, one has 
 $$
  ss\m t(uu\m v)\m=ss\m tv\m uu\m=tt\m tv\m vv\m=tv\m.
 $$
 As $tv\m=vt\m\in E(S)$, 
 $$
  tv\m=tv\m\cdot tv\m=t\cdot v\m t\cdot v\m=tt\m vv\m,
 $$
 therefore,
 $$
  \alpha(tv\m)=\alpha(tt\m)\alpha(vv\m)=aa\m bb\m
 $$
 by~\eqref{eq-aa^(-1)=alpha(ss^(-1))=alpha(tt^(-1))}--\eqref{eq-bb^(-1)=alpha(uu^(-1))=alpha(vv^(-1))}. Thus,~\eqref{eq-alpha(ss^(-1)t(uu^(-1)v)^(-1))f(s,s^(-1)t)} becomes
 $$
  aa\m bb\m f(s,s\m t)=aa\m bb\m f(u,u\m v),
 $$
 proving~\eqref{eq-af(s,s^(-1)t)b=af(u,u^(-1)v)b}.
 
 Thus, $w_{x,y}\in\cM{\cD_x\cD_{xy}}$. Clearly, it is invertible with $w\m_{x,y}a=f(s,s\m t)\m a$, $aw\m_{x,y}=af(s,s\m t)\m$, where $a$, $s$ and $t$ are as above.
\end{proof}
 
\begin{lem}\label{lem-0_x-circ-0_y}
 The pair $(\0,w)$ satisfies (iv) of Definition~\ref{defn-tw_part_act}.
\end{lem}
\begin{proof}
 Given $a\in \cD_{y\m}\cD_{y\m x\m}$, find $s\in y$ and $t\in xy$, such that 
 \begin{align}\label{eq-aa^(-1)=alpha(s^(-1)s)=alpha(t^(-1)t)}
  aa\m=\alpha(s\m s)=\alpha(t\m t). 
 \end{align}
 Then $\0_y(a)=\lambda_s(a)$, and since $\lambda_s(a)\in A_{\alpha(st\m ts\m)}=A_{\alpha((ts\m)\m ts\m)}$ with $ts\m\in x$, we have
 \begin{align}\label{eq-0_x-circ-0_y}
  \0_x\circ\0_y(a)=\lambda_{ts\m}\circ\lambda_s(a).
 \end{align}
 In view of (i), (iii) and (iv') of Definition~\ref{defn-twisted_S-mod} and~\eqref{eq-aa^(-1)=alpha(s^(-1)s)=alpha(t^(-1)t)}, the composition~\eqref{eq-0_x-circ-0_y} equals
 \begin{align*}
  \xi_{f(ts\m,s)}\circ\lambda_{ts\m s}(a)&=\xi_{f(ts\m,s)f(t,s\m s)\m}\circ\lambda_t\circ\lambda_{s\m s}(a)\\
  &=\xi_{f(ts\m,s)\alpha(ts\m st\m)}\circ\lambda_t(\alpha(s\m s)a)\\
  &=\xi_{f(ts\m,s)}\circ\lambda_t(a).
 \end{align*}

Furthermore, $\0_{xy}(a)=\lambda_t(a)$ with $\lambda_t(a)\lambda_t(a)\m=\lambda_t(aa\m)$, which on the one hand equals $\alpha(tt\m)$, and on the other hand it is $\alpha(ts\m st\m)=\alpha(ts\m(ts\m)\m)$ by~\eqref{eq-aa^(-1)=alpha(s^(-1)s)=alpha(t^(-1)t)}. Since $ts\m\in x$ and $t\in xy$, one has
 $$
  w_{x,y}\0_{xy}(a)w\m_{x,y}=\xi_{f(ts\m,(ts\m)\m t)}\circ\lambda_t(a)=\xi_{f(ts\m,st\m t)}\circ\lambda_t(a).
 $$
 But $st\m t\le s$, hence by Lemma~\ref{lem-f(s',t')}
 $$
  f(ts\m,st\m t)=\alpha(ts\m st\m)f(ts\m,s),
 $$
 the latter being $f(ts\m,s)$, as $f(ts\m,s)\in A_{\alpha(ts\m st\m)}$.
\end{proof}

\begin{cor}\label{cor-0_x(D_(x^(-1))D_y)}
 The isomorphism $\0_x$ satisfies (iii) of Definition~\ref{defn-tw_part_act}.
\end{cor}
\noindent The inclusion $\0_x(\cD_{x\m}\cD_y)\subseteq \cD_x\cD_{xy}$ is the statement of Lemma~\ref{lem-0_x(D_(x^(-1))D_y)}. Applying the same lemma to $x\m,xy\in\cG S$, we immediately get the inclusion $\0_{x\m}(\cD_x\cD_{xy})\subseteq \cD_{x\m}\cD_y$. Therefore, $\0_x\circ\0_{x\m}(\cD_x\cD_{xy})\subseteq\0_x(\cD_{x\m}\cD_y)$. By Lemmas~\ref{lem-0_x-isom} and~\ref{lem-0_x-circ-0_y} the left-hand side is $w_{x,x\m}\cD_x\cD_{xy}w\m_{x,x\m}$, which is $\cD_x\cD_{xy}$ thanks to Remark~\ref{rem-wIw^(-1)}.

\begin{lem}\label{lem-w-norm-twist}
 The pair $(\0,w)$ satisfies (v)--(vi) of Definition~\ref{defn-tw_part_act}.
\end{lem}
\begin{proof}
 Let $a\in \cD_x=\cD_x^2=\cD_x\cD_{x\cdot 1}$, that is $a\in A_{\alpha(ss\m)}$, where $s\in x=x\cdot 1$. Then $w_{x,1}a=f(s,s\m s)a$ by~\eqref{eq-w_(x,y)-from-Lambda}, which is $\alpha(ss\m)a=a$ by (iv) of Definition~\ref{defn-twisted_S-mod}. Similarly $aw_{x,1}=a$. Since $ss\m\in 1$ and $a\in A_{\alpha(ss\m(ss\m)\m)}$, it follows that $w_{1,x}a=f(ss\m,(ss\m)\m s)a=f(ss\m,s)a=\alpha(ss\m)a=a$ and similarly $aw_{1,x}=a$. Thus, $w_{x,1}$ and $w_{1,x}$ are identity on $\cD_x$, which is (v).
 
 For (vi) take $a\in \cD_{x\m}\cD_y\cD_{yz}$. There is a unique triple $s,t,u\in S$, such that $s\in x$, $t\in y$, $u\in yz$ and 
\begin{align}\label{eq-aa^(-1)=alpha(s^(-1)s)=alpha(tt^(-1))=alpha(uu^(-1))}
 aa\m=\alpha(s\m s)=\alpha(tt\m)=\alpha(uu\m).
\end{align}
Then $aw_{y,z}=af(t,t\m u)$. Since by~\eqref{eq-aa^(-1)=alpha(s^(-1)s)=alpha(tt^(-1))=alpha(uu^(-1))}
\begin{align}\label{eq-af(t,t^(-1)u)}
 af(t,t\m u)(af(t,t\m u))\m=aa\m\alpha(tt\m uu\m)=aa\m,
\end{align}
it follows that $\0_x(aw_{y,z})=\lambda_s(af(t,t\m u))$. Using~\eqref{eq-aa^(-1)=alpha(s^(-1)s)=alpha(tt^(-1))=alpha(uu^(-1))}--\eqref{eq-af(t,t^(-1)u)} we observe that
$$
 \lambda_s(af(t,t\m u))\lambda_s(af(t,t\m u))\m=\lambda_s(aa\m)=\alpha(ss\m)=\alpha(su(su)\m)
$$
with $su\in xyz$. Hence
\begin{align}\label{eq-0_x(aw_(y,z))w_(x,yz)}
 \0_x(aw_{y,z})w_{x,yz}=\lambda_s(af(t,t\m u))f(s,s\m su)=\lambda_s(af(t,t\m u))f(s,u),
\end{align}
because $s\m s=uu\m$.

Now $\0_x(a)=\lambda_s(a)$ and
$$
 \lambda_s(a)\lambda_s(a)\m=\lambda_s(aa\m)=\alpha(ss\m)=\alpha(st(st)\m)
$$
by~\eqref{eq-aa^(-1)=alpha(s^(-1)s)=alpha(tt^(-1))=alpha(uu^(-1))}. As $st\in xy$, one has that $\0_x(a)w_{x,y}=\lambda_s(a)f(s,s\m st)=\lambda_s(a)f(s,t)$, due to the fact that $s\m s=tt\m$. Finally,
$$
 \lambda_s(a)f(s,t)(\lambda_s(a)f(s,t))\m=\lambda_s(aa\m)=\alpha(st(st)\m)=\alpha(su(su)\m),
$$
where $st\in xy$ and $su\in xyz$. Consequently, taking into account that $s\m s=tt\m$, we have
\begin{align}\label{eq-0_x(a)w_(x,y)w_(xy,z)}
 \0_x(a)w_{x,y}w_{xy,z}=\lambda_s(a)f(s,t)f(st,(st)\m su)=\lambda_s(a)f(s,t)f(st,t\m u).
\end{align}

To show that~\eqref{eq-0_x(aw_(y,z))w_(x,yz)} and~\eqref{eq-0_x(a)w_(x,y)w_(xy,z)} are equal, it remains to note that
$$
 \lambda_s(f(t,t\m u))f(s,tt\m u)=f(s,t)f(st,t\m u)
$$
by (v) of Definition~\ref{defn-twisted_S-mod} and that $tt\m u=u$ thanks to~\eqref{eq-aa^(-1)=alpha(s^(-1)s)=alpha(tt^(-1))=alpha(uu^(-1))}.
\end{proof}

\begin{prop}\label{prop-A*_Lambda-S-cong-A*_Theta-G(S)}
 The pair $\Theta=(\0,w)$ is a twisted partial action of $\cG S$ on $A$. Moreover, the crossed products $A*_\Lambda S$ and $A*_\Theta\cG S$ are equivalent as extensions of $A$ by $\cG S$.
\end{prop}	
\begin{proof}
 The fact that $\Theta$ is a twisted partial action of $\cG S$ on $A$ follows from Lemmas~\ref{lem-D_x-ideal}, \ref{lem-0_x-isom}, \ref{lem-w_(x,y)-in-M(D_xD_xy)}, \ref{lem-0_x-circ-0_y}, \ref{lem-w-norm-twist} and Corollary~\ref{cor-0_x(D_(x^(-1))D_y)}.

 Let $a\delta_x\in A*_\Theta\cG S$. Since $a\in \cD_x$, there is a unique $s\in x$, such that $a\in A_{\alpha(ss\m)}$. Set $\varphi(a\delta_x)=a\delta_s$. This is a well-defined bijection $A*_\Theta\cG S\to A*_\Lambda S$ with $\varphi\m(a\delta_s)=a\delta_{\sigma^\natural(s)}$. 
 
 Recall that $A*_\Theta\cG S$ and $A*_\Lambda S$ can be regarded as extensions of $A$ by $\cG S$ with embeddings $i(a)=a\delta_1$, $i'(a)=a\delta_{\alpha\m(aa\m)}$ and projections $j(a\delta_x)=x$, $j'(a\delta_s)=\sigma^\natural(s)$, respectively. Observe that
 $$
  \varphi\circ i(a)=\varphi(a\delta_1)=a\delta_e,
 $$
 where $e\in E(S)$ with $\alpha(e)=aa\m$. It follows that $e=\alpha\m(aa\m)$, i.\,e. $\varphi\circ i(a)=i'(a)$. Moreover, for $a$, $x$ and $s$ as above:
 $$
  j'\circ\varphi(a\delta_x)=j'(a\delta_s)=\sigma^\natural(s),
 $$
 which is $x=j(a\delta_x)$, as $s\in x$.
 
 It remains to show that $\varphi$ is a homomorphism. Given $a\delta_x,b\delta_y\in A*_\Theta\cG S$, we immediately have
 $$
  \varphi(a\delta_x)\varphi(b\delta_y)=a\delta_s\cdot b\delta_t=a\lambda_s(b)f(s,t)\delta_{st},
 $$
 where $s\in x$, $t\in y$, $a\in A_{\alpha(ss\m)}$ and $b\in A_{\alpha(tt\m)}$. On the other hand
 $$
  \varphi(a\delta_x\cdot b\delta_y)=\varphi(\0_x(\0\m_x(a)b)w_{x,y}\delta_{xy})=\0_x(\0\m_x(a)b)w_{x,y}\delta_u
 $$
 with $u\in xy$, such that $\0_x(\0\m_x(a)b)w_{x,y}\in A_{\alpha(uu\m)}$. It is enough to prove that $\0_x(\0\m_x(a)b)w_{x,y}=a\lambda_s(b)f(s,t)$. Then automatically $u=st$, as $st\in xy$ and $a\lambda_s(b)f(s,t)\in A_{\alpha(st(st)\m)}$.
 
 By Lemma~\ref{lem-0_x-inv} one has $\0\m_x(a)=w\m_{x\m,x}\0_{x\m}(a)w_{x\m,x}$. Since $a\in A_{\alpha(ss\m)}$ and $s\m\in x\m$, then
 $$
  \0_{x\m}(a)=\lambda_{s\m}(a)\in A_{\alpha(s\m s)}=A_{\alpha((s\m s)\m s\m s)}.
 $$
 Therefore,
 \begin{align*}
  w\m_{x\m,x}\0_{x\m}(a)w_{x\m,x}&=f(s\m,(s\m)\m s\m s)\m\lambda_{s\m}(a)f(s\m,(s\m)\m s\m s)\\
  &=f(s\m,s)\m\lambda_{s\m}(a)f(s\m,s),
 \end{align*}
 which is $\overline{\lambda_s}(a)\in A_{\alpha(s\m s)}$ in view of Lemma~\ref{lem-lambda_s_satisfies_(i)}. It follows that
 $$
  \0\m_x(a)b=\overline{\lambda_s}(a)b\in A_{\alpha(s\m stt\m)}=A_{\alpha((stt\m)\m stt\m)},
 $$
 where $stt\m\in x$. Hence $\0_x(\0\m_x(a)b)=\lambda_{stt\m}(\overline{\lambda_s}(a)b)$. By (i)--(iii) and (iv') of Definition~\ref{defn-twisted_S-mod}
 $$
  \lambda_{stt\m}=\xi_{f(s,tt\m)\m}\circ\lambda_s\circ\lambda_{tt\m}=\alpha(stt\m s\m)\lambda_s.
 $$
 Thus, taking into account Lemma~\ref{lem-lambda_s_satisfies_(i)}
 $$
  \0_x(\0\m_x(a)b)=\alpha(stt\m s\m)\lambda_s(\overline{\lambda_s}(a)b)=\alpha(stt\m s\m)\alpha(ss\m)a\lambda_s(b)=a\lambda_s(b),
 $$
 as $\lambda_s(b)\in A_{\alpha(stt\m s\m)}$. Now observing that $stt\m s\m=stt\m(stt\m)\m=st(st)\m$ and $stt\m\in x$, $st\in xy$, we conclude that
 $$
  \0_x(\0\m_x(a)b)w_{x,y}=a\lambda_s(b)f(stt\m,tt\m s\m st)=a\lambda_s(b)f(stt\m,s\m st).
 $$
 By Lemma~\ref{lem-f(s',t')}
 $$
  f(stt\m,s\m st)=\alpha(stt\m s\m)f(s,t)=f(s,t),
 $$  as $stt\m\cdot s\m st=st.$
\end{proof}

\begin{thrm}\label{thrm-adm-ext-equiv-crossed-prod}
 Let $A$ be a semilattice of groups and $G$ a group. For any admissible extension $U$ of $A$ by $G$ there exists a twisted partial action $\Theta$ of $G$ on $A$, such that $U$ is equivalent to $A*_\Theta G$.
\end{thrm}
\begin{proof}
 Let $A\overset{i}{\to}U\overset{j}{\to}G$ be an admissible extension. By Proposition~\ref{prop-ext-A-G-to-ext-A-S} there exist an $E$-unitary inverse semigroup $S$ and epimorphisms $\pi:U\to S$, $\kappa:S\to G$, such that $A\overset{i}{\to}U\overset{\pi}{\to}S$ is an extension (in the sense of Lausch) of $A$ by $S$, $j=\kappa\circ\pi$ and $\ker\kappa=\sigma$. Since $U$ is admissible, we may assume that $\pi$ has an order-preserving transversal $\rho$. 
 
 Denote by $\Lambda=(\alpha,\lambda,f)$ the twisted $S$-module structure on $A$ that comes from the extension $A\overset{i}{\to}U\overset{\pi}{\to}S$ and transversal $\rho$. By Remark~\ref{rem-A*_Lambda-S-equiv-U} the extensions $U$ and $A*_\Lambda S$ of $A$ by $S$ are equivalent. With the help of $\kappa$ they can be seen as equivalent extensions of $A$ by $G$.
 
 Observe that $\Lambda$ is a Sieben twisted $S$-module structure thanks to Proposition~\ref{prop-order-pres-rho}. Identifying $G$ with $\cG S$ and using Proposition~\ref{prop-A*_Lambda-S-cong-A*_Theta-G(S)}, we construct a twisted partial action $\Theta$ of $G$ on $A$, such that $A*_\Lambda S$, and hence by transitivity $U$, is equivalent to $A*_\Theta G$ as an extension of $A$ by $G$. 
\end{proof}

\begin{prop}\label{prop-equiv-Lambdas=>equiv-Thetas}
 Let $\Lambda=(\alpha,\lambda,f)$ and $\Lambda'=(\alpha',\lambda',f')$ be Sieben twisted $S$-module structures. If  $\Lambda$  is equivalent to $\Lambda' ,$  then the corresponding  $\Theta=(\0,w)$ and $\Theta'=(\0',w')$ are  equivalent.
\end{prop}
\begin{proof}
 Since $\alpha'=\alpha$, we see from~\eqref{eq-D_x-from-Lambda} that $\cD'_x=\cD_x$.
 
 Let $g:S\to A$ be the map from Proposition~\ref{prop-equiv-tw-S-mod}, which determines the equivalence of $\Lambda$ and $\Lambda'$. Writing (iii) of Proposition~\ref{prop-equiv-tw-S-mod} for the pair $(e,s)$, where $e\in E(S)$, and using (i), (iv') of Definition~\ref{defn-twisted_S-mod} and Remark~\ref{rem-g(e)}, we observe that 
 \begin{align}\label{eq-g(es)}
  g(es)=\alpha(e)g(s).
 \end{align}
 
 Given $a\in\cD_x$, set $\e_xa=g(s)a$ and $a\e_x=ag(s)$, where $s\in x$ with $\alpha(ss\m)=aa\m$. We first show that $\e_x\in\cM{\cD_x}$. If $b$ is an other element of $\cD_x$, $t\in x$ and $\alpha(tt\m)=bb\m$, then
 $$
  ab(ab)\m=bb\m aa\m=\alpha(tt\m ss\m)=\alpha(tt\m s(tt\m s)\m),
 $$
 where $tt\m s\in x$. Hence, using \eqref{eq-g(es)},
 $$
  \e_x(ab)=g(tt\m s)ab=\alpha(tt\m)g(s)ab=g(s)ab=(\e_xa)b.
 $$
 The equality $(ab)\e_x=a(b\e_x)$ is proved similarly. As to $(a\e_x)b=a(\e_xb)$, it is equivalent to $ag(s)b=ag(t)b$. Note that, since $S$ is $E$-unitary and $(s,t)\in\sigma$, one has $s\m t,st\m\in E(S)$, so
 $$
  ss\m t=s\cdot s\m t\cdot t\m t=s\cdot t\m t\cdot s\m t=st\m\cdot ts\m\cdot t=ts\m t=tt\m s.
 $$
Using this and~\eqref{eq-g(es)}, we get
\begin{align*}
 ag(s)b&=ag(s)bb\m b=abb\m g(s)b=a\alpha(tt\m)g(s)b=ag(tt\m s)b\\
 &=ag(ss\m t)b=a\alpha(ss\m)g(t)b=a\cdot aa\m g(t)b=ag(t)b. 
\end{align*}
Obviously, $\e_x$ is invertible with $\e\m_xa=g(s)\m a$ and $a\e\m_x=ag(s)\m$, where $\alpha(ss\m)=aa\m$.

Let $a\in\cD_{x\m}$ and $s\in x$ with $\alpha(s\m s)=aa\m$. Then by~\eqref{eq-0_x-from-Lambda} and (ii) of Proposition~\ref{prop-equiv-tw-S-mod}
$$
 \0'_x(a)=\lambda'_s(a)=g(s)\lambda_s(a)g(s)\m=\e_x\0_x(a)\e\m_x,
$$
as $\lambda_s(a)\lambda_s(a)\m=\lambda_s(aa\m)=\lambda_s(\alpha(s\m s))=\alpha(ss\m)=g(s)g(s)\m$.

Given $a\in\cD_{x\m}\cD_y$, choose $s\in x$ and $t\in y$ with $\alpha(s\m s)=\alpha(tt\m)=a a\m$. Then $\0'_x(a)=\lambda'_s(a)$. Since
$$
\lambda'_s(a)\lambda'_s(a)\m=\lambda'_s(aa\m)=\alpha(ss\m)=\alpha(stt\m s\m)
$$
with $s\in x$ and $st\in xy$, it follows from~\eqref{eq-w_(x,y)-from-Lambda} that $\0'_x(a)w'_{x,y}=\lambda'_s(a)f'(s,s\m st)$. The latter belongs to $A_{\alpha(stt\m s\m)}$, so in view of (iii) of Proposition~\ref{prop-equiv-tw-S-mod},~\eqref{eq-g(es)} and (ii) of Definition~\ref{defn-twisted_S-mod}
\begin{align*}
 \0'_x(a)w'_{x,y}\e_{xy}&=\lambda'_s(a)f'(s,s\m st)g(st)\\
 &=g(s)\lambda_s(a)g(s)\m g(s)\lambda_s(g(s\m st))f(s,s\m st)\\
 &=g(s)\lambda_s(a\alpha(s\m s)g(t))f(s,s\m st)\\
 &=g(s)\alpha(ss\m)\lambda_s(ag(t))f(s,s\m st)\\
 &=g(s)\lambda_s(ag(t))f(s,s\m st)\\
 &=\e_x\0_x(a\e_y)w_{x,y}.
\end{align*}
\end{proof}

\section{\texorpdfstring{Transversals of $j$ and the corresponding twisted partial actions}{Transversals of j and the corresponding twisted partial actions}}\label{sec-transversal}

Our next purpose is to find explicit formulas for $\Theta$ which can be obtained directly from an extension $U$ of $A$ by $G$ (and do not involve $S$ and $\Lambda$).

\begin{lem}\label{lem-ker-pi}
 Let $A\overset{i}{\to}U\overset{j}{\to}G$ be an extension, $\pi:U\to S$, $\kappa:S\to G$ the corresponding epimorphisms from Proposition~\ref{prop-ext-A-G-to-ext-A-S}. Then
 $$
  \pi(u)=\pi(v)\iff j(u)=j(v)\mbox{ and }uu\m=vv\m.
 $$
\end{lem}
\begin{proof}
 Note that 
 $$
  j(u)=j(v)\iff \kappa\circ\pi(u)=\kappa\circ\pi(v)\iff (\pi(u),\pi(v))\in\sigma,
 $$
 as $\ker\kappa=\sigma$. Since $\pi$ is idempotent-separating, then
 $$
  uu\m=vv\m\iff\pi(uu\m)=\pi(vv\m)\iff\pi(u)\pi(u)\m=\pi(v)\pi(v)\m.
 $$
 It remains to apply Corollary~\ref{cor-s=t-in-E-unitary}.
\end{proof}

Lemma~\ref{lem-ker-pi} shows that with each class of $\ker\pi$ one can associate a pair $(x,e)\in G\times E(U)$, where $x=j(u)$ and $e=uu\m$ for arbitrary element $u$ of the class. Conversely, given $(x,e)\in G\times E(U)$, the set of all $u\in U$, such that $j(u)=x$ and $uu\m=e$, is either empty, or forms a class of $\ker\pi$.

\begin{defn}\label{defn-tau(x,e)}
 Let $A\overset{i}{\to}U\overset{j}{\to}G$ be an extension. A {\it transversal} of $j$ is a partial map $\tau:G\times E(U)\to U$, such that
 \begin{enumerate}
  \item $\tau(x,e)$ is defined if and only if the set $U(x,e)=\{u\in U\mid j(u)=x,\ \ uu\m=e\}$ is nonempty;
  \item $\tau(x,e)\in U(x,e)$, whenever defined;
  \item $\tau(1,e)=e$.
 \end{enumerate}

\end{defn}

\begin{rem}\label{rem-dom-tau}
 Note that all the transversals of $j$ have the same domain $\{(x,e)\in G\times E(U)\mid U(x,e)\ne\emptyset\}$, which depends only on the extension.
\end{rem}

\begin{prop}\label{prop-rho-and-tau}
 Under the conditions of Lemma~\ref{lem-ker-pi} there is a one-to-one correspondence between transversals of $\pi$ and transversals of $j$. 
\end{prop}
\begin{proof}
 Let $\rho$ be a transversal of $\pi$ and $u\in U(x,e)\ne\emptyset$. We shall prove that 
\begin{align}\label{eq-tau-from-rho}
 \tau(x,e)=\rho\circ\pi(u) 
\end{align}
is a transversal of $j$. First of all note that $\rho\circ\pi(u)$ does not depend on the choice of $u$ thanks to Lemma~\ref{lem-ker-pi}. Furthermore,
$$
 j\circ\tau(x,e)=j\circ\rho\circ\pi(u)=\kappa\circ\pi\circ\rho\circ\pi(u)=\kappa\circ\pi(u)=j(u)=x
$$
and by~\eqref{eq-rho(s)rho(s^(-1))}
$$
 \tau(x,e)\tau(x,e)\m=\rho(\pi(u))\rho(\pi(u))\m=\rho(\pi(uu\m))=\pi\m\circ\pi(uu\m)=uu\m,
$$
so $\tau(x, e)\in U(x,e)$. If $x=1$, then one has
$\kappa\circ\pi(u)=j(u)=1$ and hence $\pi(u)\in E(S)$, as $S$ is $E$-unitary and $\ker\kappa=\sigma$. Therefore
$$
 \tau(1,e)=\rho(\pi(u))=\rho(\pi(u)\pi(u)\m)=\pi\m\circ\pi(uu\m)=uu\m=e.
$$

Conversely, given a transversal $\tau$ of $j$, we show that 
\begin{align}\label{eq-rho-from-tau}
 \rho(s)=\tau(\kappa(s),\pi\m(ss\m)),
\end{align}
$s\in S$, is a transversal of $\pi$. Indeed, if $s=\pi(u)$, then
\begin{align}\label{eq-U(j(u),uu^(-1))}
 U(\kappa(s),\pi\m(ss\m))=U(\kappa\circ\pi(u),\pi\m(\pi(u)\pi(u)\m))=U(j(u),uu\m), 
\end{align}
which is nonempty, because $u\in U(j(u),uu\m)$. Hence, $\rho(s)$ is well-defined. Now
$$
 \kappa\circ\pi\circ\rho(s)=j\circ\tau(\kappa(s),\pi\m(ss\m))=\kappa(s),
$$ 
and since
$$
 \rho(s)\rho(s)\m=\tau(\kappa(s),\pi\m(ss\m))\tau(\kappa(s),\pi\m(ss\m))\m=\pi\m(ss\m),
$$
one has $\pi(\rho(s))\pi(\rho(s))\m=ss\m$. Then $\pi\circ\rho(s)=s$ in view of Corollary~\ref{cor-s=t-in-E-unitary} and the fact that $\ker\kappa=\sigma$. Finally, for $e\in E(S)$:
$$
 \rho(e)=\tau(1,\pi\m(e))=\pi\m(e)\in E(U).
$$

We now show that the maps $\rho\mapsto\tau$ and $\tau\mapsto\rho$ are mutually inverse. If $\rho\mapsto\tau\mapsto\rho'$, then by~\eqref{eq-tau-from-rho}--\eqref{eq-rho-from-tau}
$$
 \rho'(s)=\tau(\kappa(s),\pi\m(ss\m))=\rho\circ\pi(u),
$$
where $u\in U(\kappa(s),\pi\m(ss\m))$. Taking $v\in U$ with $\pi(v)=s$, we observe by~\eqref{eq-U(j(u),uu^(-1))} that
$U(\kappa(s),\pi\m(ss\m))=U(j(v),vv\m)$, so $j(u)=j(v)$ and $uu\m=vv\m$. Applying Lemma~\ref{lem-ker-pi}, we get $\pi(u)=\pi(v)=s$. Thus, $\rho'(s)=\rho(s)$.

Now let $\tau\mapsto\rho\mapsto\tau'$. According to~\eqref{eq-tau-from-rho}--\eqref{eq-rho-from-tau} $\tau'(x,e)$ equals
\begin{align}\label{eq-rho-circ-pi}
 \rho\circ\pi(u)=\tau(\kappa\circ\pi(u),\pi\m(\pi(u)\pi(u)\m))=\tau(j(u),uu\m), 
\end{align}
where $u\in U(x,e)$. But $j(u)=x$ and $uu\m=e$ by the definition of $U(x,e)$, yielding $\tau'(x,e)=\tau(x,e)$.
\end{proof}

\begin{rem}\label{rem-tau-preserves-order}
 The transversal $\rho$ of $\pi$ is order-preserving if and only if the corresponding transversal $\tau$ of $j$ satisfies
 \begin{align}\label{eq-tau-preserves-order}
  u\le v\impl\tau(j(u),uu\m)\le\tau(j(v),vv\m).
 \end{align}
\end{rem}
\noindent For it is easily seen that $\rho$ preserves the order, whenever $\rho\circ\pi$ does. It remains to use~\eqref{eq-rho-circ-pi}.

\begin{defn}
 A transversal $\tau$ of $j$, for which~\eqref{eq-tau-preserves-order} holds, will be called {\it order-preserving}.
\end{defn}

\begin{rem}\label{rem-tau-for-adm-ext}
 An extension $A\overset{i}{\to}U\overset{j}{\to}G$ is admissible if and only if $j$ has an order-preserving transversal.
\end{rem}


\begin{prop}\label{prop-(0,w)-from-tau}
 Let $A\overset{i}{\to}U\overset{j}{\to}G$ be an admissible extension and $\tau$ an order-preserving transversal of $j$. Then the pair $(U,\tau)$ induces a twisted partial action $\Theta=(\0,w)$ of $G$ on $A$ by the formulas
 \begin{align}
  \cD_x&=\bigsqcup_{U(x,e)\ne\emptyset}A_{i\m(e)};\label{eq-D_x-from-tau}\\
  \0_x(a)&=i\m(\xi_{\tau(x,\tau(x\m,i(aa\m))\m\tau(x\m,i(aa\m)))}\circ i(a)),\ \ a\in\cD_{x\m};\label{eq-0_x-from-tau}\\
  w_{x,y}a&=\omega(x,y,a)a,\ \ aw_{x,y}=a\omega(x,y,a),\mbox{ where }a\in\cD_x\cD_{xy}\mbox{ and}\notag\\
  \omega(x,y,a)&=i\m(\tau(x,i(aa\m))\tau(y,\xi_{\tau(x,i(aa\m))\m}\circ i(aa\m))
  \tau(xy,i(aa\m))\m).\label{eq-w_(x,y)-from-tau}
 \end{align}
\end{prop}
\begin{proof}
 Find a pair of epimorphisms $\pi:U\to S$, $\kappa:S\to G$ as in Proposition~\ref{prop-ext-A-G-to-ext-A-S} and denote by $\rho$ the transversal of $\pi$ corresponding to $\tau$. Let $\Lambda=(\alpha,\lambda,f)$ be the twisted $S$-module structure on $A$ coming from $A\overset{i}{\to}U\overset{\pi}{\to}S$ and $\rho$. We shall also identify $x\in G$ with the $\sigma$-class $\kappa\m(x)$. 
 
 By~\eqref{eq-alpha=i^(-1)rho_E(S)} and~\eqref{eq-D_x-from-Lambda}
 $$
  \cD_x=\bigsqcup_{\kappa(s)=x}A_{\alpha(ss\m)}=\bigsqcup_{\kappa(s)=x}A_{i\m\circ\rho(ss\m)}=\bigsqcup_{\kappa(s)=x}A_{i\m\circ\pi\m(ss\m)}.
 $$
 For~\eqref{eq-D_x-from-tau} we need to show that 
 $$
  \{\pi\m(ss\m)\mid\kappa(s)=x\}=\{e\in E(U)\mid U(x,e)\ne\emptyset\}.
 $$
 Indeed, if $\kappa(s)=x$, then taking $u\in U$ with $\pi(u)=s$, we have $j(u)=\kappa\circ\pi(u)=\kappa(s)=x$ and $uu\m=\pi\m(ss\m)$, so $u\in U(x,\pi\m(ss\m))\ne\emptyset$. Conversely, let $u\in U(x,e)\ne\emptyset$. Then for $s=\pi(u)$ one has $\kappa(s)=\kappa\circ\pi(u)=j(u)=x$ and $\pi\m(ss\m)=uu\m=e$.
 
 Given $a\in\cD_{x\m}$, we have by~\eqref{eq-defn_of_nu_u}--\eqref{eq-lambda=nu-rho} and~\eqref{eq-0_x-from-Lambda}
 $$
  \0_x(a)=i\m(\xi_{\rho(s)}\circ i(a)),
 $$
 where $s$ is a unique element of $S$, satisfying $\kappa(s)=x$ and $\alpha(s\m s)=aa\m$. According to~\eqref{eq-rho-from-tau} and~\eqref{eq-alpha=i^(-1)rho_E(S)}--\eqref{eq-rho(s)rho(s^(-1))}
 \begin{align*}
  \rho(s)&=\tau(\kappa(s),\pi\m(ss\m))=\tau(x,\rho(ss\m))=\tau(x,\rho(s\m)\m\rho(s\m)),\\
  \rho(s\m)&=\tau(\kappa(s\m),\pi\m(s\m s))=\tau(x\m,\rho(s\m s))=\tau(x\m,i(aa\m)),
 \end{align*}
 proving~\eqref{eq-0_x-from-tau}.
 
 It follows from~\eqref{eq-rho(s)rho(t)} and~\eqref{eq-w_(x,y)-from-Lambda} that $w_{x,y}$ acts on $a\in\cD_x\cD_{xy}$ as the multiplication by
 $$
  f(s,s\m t)=i\m(\rho(s)\rho(s\m t)\rho(ss\m t)\m),
 $$
 where $\kappa(s)=x$, $\kappa(t)=xy$ and $\alpha(ss\m)=\alpha(tt\m)=aa\m$.  Notice that, given $s\in S$ and $e\in E(S)$, one has that $\rho(s\m es)=\rho(s)\m\rho(e)\rho(s),$ because the images of these idempotents under $\pi$ coincide. Using~\eqref{eq-alpha=i^(-1)rho_E(S)} and \eqref{eq-rho-from-tau}, we obtain
 \begin{align*}
  \rho(s)&=\tau(\kappa(s),\pi\m(ss\m))=\tau(x,\rho(ss\m))=\tau(x,i(aa\m)),\\
  \rho(s\m t)&=\tau(\kappa(s\m t),\pi\m(s\m tt\m s))=\tau(y,\rho(s\m tt\m s))\\
  &=\tau(y,\rho(s)\m\rho(tt\m)\rho(s))=\tau(y,\rho(s)\m i(aa\m)\rho(s)),\\
  \rho(ss\m t)&=\tau(\kappa(ss\m t),\pi\m(ss\m tt\m))=\tau(xy,\rho(ss\m tt\m))\\
  &=\tau(xy,\rho(ss\m)\rho(tt\m))=\tau(xy,i(aa\m)).
 \end{align*}
 This proves~\eqref{eq-w_(x,y)-from-tau}.
\end{proof}

\begin{prop}\label{prop-tau-for-A*_Theta-G}
 Let $\Theta=(\0,w)$ be a twisted partial action of $G$ on $A$ and $A\overset{i}{\to}A*_\Theta G\overset{j}{\to}G$ the corresponding (admissible) extension of $A$ by $G$. Then the partial map $\tau:G\times E(A*_\Theta G)\to A*_\Theta G$ with 
 \begin{align}
  \dom\tau&=\{(x,e\delta_1)\mid e\in\cD_x\},\label{eq-dom-tau-for-A*_Theta-G}\\
  \tau(x,e\delta_1)&=e\delta_x,\ \ (x,e\delta_1)\in\dom\tau,\label{eq-tau-for-A*_Theta-G}
 \end{align}
 is an order-preserving transversal of $j$. Moreover, the twisted partial action of $G$ on $A$ induced by $(A*_\Theta G,\tau)$ is $\Theta$.
\end{prop}
\begin{proof}
 For an arbitrary $a\delta_y\in A*_\Theta G$ we have $j(a\delta_y)=y$ and $a\delta_y(a\delta_y)\m=aa\m\delta_1$ by Corollary~\ref{cor-sdelta_x(sdelta_x)^(-1)}. Then
$$
 U(x,e\delta_1)=\{a\delta_y\in A*_\Theta G\mid y=x,\ \ aa\m=e\}=
 \begin{cases}
  A_e\delta_x, & e\in\cD_x,\\
  \emptyset, & e\not\in\cD_x,
 \end{cases}
$$
so~\eqref{eq-dom-tau-for-A*_Theta-G} is (i) of Definition~\ref{defn-tau(x,e)}. Given $(x,e\delta_1)\in\dom\tau$, observe from~\eqref{eq-tau-for-A*_Theta-G} that
\begin{align*}
 j\circ\tau(x,e\delta_1)&=j(e\delta_x)=x,\\
 \tau(x,e\delta_1)\tau(x,e\delta_1)\m&=e\delta_x(e\delta_x)\m=e\delta_1,
\end{align*}
consequently $\tau(x,e\delta_1)\in U(x,e\delta_1)$ proving (ii) of Definition~\ref{defn-tau(x,e)}. Since, obviously, $\tau(1,e\delta_1)=e\delta_1$, we conclude that $\tau$ satisfies (iii) of Definition~\ref{defn-tau(x,e)}, and hence it is a transversal of $j$. Moreover, $\tau$ is order-preserving, as $u=a\delta_x\le b\delta_y=v$ means that $x=y$ and $a\le b$, so
$$
 \tau(j(u),uu\m)=\tau(x,aa\m\delta_1)=aa\m\delta_x\le bb\m\delta_y=\tau(j(v),vv\m).
$$

Let $\Theta'=(\0',w')$ be the twisted partial action of $G$ on $A$ induced by $(A*_\Theta G,\tau)$ as in Proposition~\ref{prop-(0,w)-from-tau}. Observe using~\eqref{eq-D_x-from-tau} and~\eqref{eq-dom-tau-for-A*_Theta-G} that
$$
 \cD'_x=\bigsqcup_{U(x,e\delta_1)\ne\emptyset}A_{i\m(e\delta_1)}=\bigsqcup_{e\in\cD_x}A_e=\cD_x,
$$
as an ideal $I$ in a semilattice of groups is the (disjoint) union of the group components corresponding to $E(I)$. 

Take $a\in\cD_{x\m}$. We first calculate 
$$
 \tau(x\m,i(aa\m))=\tau(x\m,aa\m\delta_1)=aa\m\delta_{x\m}.
$$
Hence by Corollary~\ref{cor-sdelta_x(sdelta_x)^(-1)} and Remark~\ref{rem-part-act-on-E(S)}
$$
 \tau(x\m,i(aa\m))\m\tau(x\m,i(aa\m))=\0\m_{x\m}(aa\m)\delta_1=\0_x(aa\m)\delta_1.
$$
It follows that
$$
 \tau(x,\tau(x\m,i(aa\m))\m\tau(x\m,i(aa\m)))=\tau(x,\0_x(aa\m)\delta_1)=\0_x(aa\m)\delta_x.
$$
Since by~\eqref{eq-s-delta_x-inv} and Remarks~\ref{rem-part-act-on-E(S)} and~\ref{rem-M(S)-for-comm-S}
$$
 (\0_x(aa\m)\delta_x)\m=w\m_{x\m,x}\0_{x\m}\circ \0_x(aa\m)\delta_{x\m}=aa\m w\m_{x\m,x}\delta_{x\m},
$$
then $\xi_{\tau(x,\tau(x\m,i(aa\m))\m\tau(x\m,i(aa\m)))}\circ i(a)$ equals
\begin{align*}
 \0_x(aa\m)\delta_x\cdot a\delta_1\cdot aa\m w\m_{x\m,x}\delta_{x\m}&=\0_x(aa\m)\delta_x\cdot aw\m_{x\m,x}\delta_{x\m}\\
 &=\0_x(\0\m_x(\0_x(aa\m))aw\m_{x\m,x})w_{x,x\m}\delta_1\\
 &=\0_x(aw\m_{x\m,x})w_{x,x\m}\delta_1\\
 &=\0_x(a)w\m_{x,x\m} w_{x,x\m}\delta_1\\
 &=\0_x(a)\delta_1.
\end{align*}
Here we used~\eqref{eq-0_x(sw_(x^(-1)x))}. Thus, in view of~\eqref{eq-0_x-from-tau}
$$
 \0'_x(a)=i\m(\0_x(a)\delta_1)=\0_x(a).
$$

As to $w'_{x,y}$, we observe that $\tau(x,i(aa\m))=aa\m\delta_x$ and $\tau(xy,i(aa\m))=aa\m\delta_{xy}$, where $a\in\cD_x\cD_{xy}$. Using~\eqref{eq-s-delta_x-inv}, \eqref{eq-0_x(sw_(x^(-1)x))}, Lemma~\ref{lem-0_x-inv}, Remark~\ref{rem-M(S)-for-comm-S} and (iv) of Definition~\ref{defn-tw_part_act}, we have
\begin{align*}
 \xi_{\tau(x,i(aa\m))\m}\circ i(aa\m)&=w\m_{x\m,x}\0_{x\m}(aa\m)\delta_{x\m}\cdot aa\m\delta_1\cdot aa\m\delta_x\\
 &=w\m_{x\m,x}\0_{x\m}(aa\m)\delta_{x\m}\cdot aa\m\delta_x\\
 &=\0_{x\m}(\0\m_{x\m}(w\m_{x\m,x}\0_{x\m}(aa\m))aa\m)w_{x\m,x}\delta_1\\
 &=\0_{x\m}(w\m_{x,x\m}\0_x(\0_{x\m}(aa\m)w\m_{x\m,x})w_{x,x\m}aa\m)w_{x\m,x}\delta_1\\
 &=\0_{x\m}(w\m_{x,x\m}\0_x(\0_{x\m}(aa\m))aa\m)w_{x\m,x}\delta_1\\
 &=\0_{x\m}(aa\m w\m_{x,x\m})w_{x\m,x}\delta_1\\
 &=\0_{x\m}(aa\m)\delta_1.
\end{align*}
Therefore,
$$
 \tau(y,\xi_{\tau(x,i(aa\m))\m}\circ i(aa\m))=\0_{x\m}(aa\m)\delta_y.
$$
Since $\0\m_x(aa\m)=\0_{x\m}(aa\m)$ thanks to Remark~\ref{rem-part-act-on-E(S)}, then
\begin{align*}
 \tau(x,i(aa\m))\tau(y,\xi_{\tau(x,i(aa\m))\m}\circ i(aa\m))&=aa\m\delta_x\cdot\0_{x\m}(aa\m)\delta_y\\
 &=\0_x(\0\m_x(aa\m))w_{x,y}\delta_{x y}\\
 &=aa\m w_{x,y}\delta_{xy}\\
 &=aa\m w_{x,y}\delta_1\cdot aa\m\delta_{xy}.
\end{align*}
According to~\eqref{eq-w_(x,y)-from-tau} and Corollary~\ref{cor-sdelta_x(sdelta_x)^(-1)}
\begin{align*}
 \omega(x,y,a)&=i\m(aa\m w_{x,y}\delta_1\cdot aa\m\delta_{xy}\cdot(aa\m\delta_{xy})\m)\\
 &=i\m(aa\m w_{x,y}\delta_1\cdot aa\m\delta_1)\\
 &=i\m(aa\m w_{x,y}\delta_1)\\
 &=aa\m w_{x,y}.
\end{align*}
Hence, in view of Remark~\ref{rem-M(S)-for-comm-S}
$$
 w'_{x,y}a=\omega(x,y,a)a=aa\m w_{x,y}a=w_{x,y}aa\m a=w_{x,y}a.
$$
Similarly
$$
 aw'_{x,y}=a\omega(x,y,a)=a\cdot aa\m w_{x,y}=aw_{x,y}.
$$
\end{proof}

\section{From twisted partial actions to Sieben twisted modules}\label{sec-Lambda^Theta}

Throughout this section $G$ is a group, $A$ is a semilattice of groups and $\Theta=(\0,w)$ is a twisted partial action of $G$ on $A$. 

As one knows from Remark~\ref{rem-part-act-on-E(S)}, $\0$ restricts to a partial action of $G$ on $E(A)$, and moreover by Proposition~\ref{prop-A*_Theta-G-is-ext} the crossed product $A*_\Theta G$ is an extension of $A$ by the $E$-unitary semigroup $S=E(A)*_\0 G$. Here $i(a)=a\delta_1$ and $j(a\delta_x)=aa\m\delta_x$. Recall that $j$ has the trivial order-preserving transversal $\rho(e\delta_x)=e\delta_x$. We shall also use the epimorphism $\kappa:S\to G$ mapping $a\delta_x$ to $x$ and satisfying $\ker\kappa=\sigma$.

Denote by $\Lambda=(\alpha,\lambda,f)$ the twisted $S$-module structure on $A$ that comes from the extension $A\overset{i}{\to}A*_\Theta G\overset{j}{\to}S$ and transversal $\rho$. In the next lemma we give precise formulas for $\Lambda$ in terms of $\Theta$.

\begin{lem}\label{lem-Lambda^Theta}
 For $e\delta_1\in E(S)$, $s=e'\delta_x,t=e''\delta_y\in S$ and $a\in A$ one has
 \begin{align}
  \alpha(e\delta_1)&=e,\label{eq-alpha(edelta_1)}\\
  \lambda_s(a)&=\0_x(\0\m_x(e')a)=\0_{\kappa(s)}(\alpha(s\m s)a),\label{eq-lambda_s(a)}\\
  f(s,t)&=\0_x(\0\m_x(e')e'')w_{x,y}=\alpha(stt\m s\m)w_{\kappa(s),\kappa(t)}.\label{eq-f(s,t)}
 \end{align}
\end{lem}
\begin{proof}
 In view of~\eqref{eq-alpha=i^(-1)rho_E(S)}
 $$
  \alpha(e\delta_1)=i\m\circ\rho(e\delta_1)=i\m(e\delta_1)=e,
 $$
 whence~\eqref{eq-alpha(edelta_1)}. 
 
 To prove~\eqref{eq-lambda_s(a)}, use~\eqref{eq-defn_of_nu_u}--\eqref{eq-lambda=nu-rho}, \eqref{eq-s-delta_x-inv} and Remark~\ref{rem-part-act-on-E(S)}:
 \begin{align*}
  \lambda_s(a)&=i\m(\rho(s)i(a)\rho(s)\m)\\
  &=i\m(e'\delta_x\cdot a\delta_1\cdot(e'\delta_x)\m)\\
  &=i\m(e'\delta_x\cdot a\delta_1\cdot w\m_{x\m,x}\0_{x\m}(e')\delta_{x\m})\\
  &=i\m(e'\delta_x\cdot a\0_{x\m}(e')w\m_{x\m,x}\delta_{x\m})\\
  &=i\m(\0_x(a\0_{x\m}(e')w\m_{x\m,x})w_{x,x\m}\delta_1)\\
  &=\0_x(a\0_{x\m}(e')w\m_{x\m,x})w_{x,x\m}.
 \end{align*}
 Thanks to~\eqref{eq-0_x(sw_(x^(-1)x))}
 $$
  \0_x(a\0_{x\m}(e')w\m_{x\m,x})w_{x,x\m}=\0_x(a\0_{x\m}(e'))w\m_{x,x\m}w_{x,x\m}=\0_x(a\0_{x\m}(e')),
 $$
 the latter being $\0_x(\0\m_x(e')a)$ by  Remark~\ref{rem-part-act-on-E(S)} 
 and the fact that $E(A)\subseteq C(A)$. For the second equality of~\eqref{eq-lambda_s(a)} we use Corollary~\ref{cor-sdelta_x(sdelta_x)^(-1)}.
 
 According to~\eqref{eq-rho(s)rho(t)}, for~\eqref{eq-f(s,t)} we need to multiply $\rho(s)$ by $\rho(t)$:
 \begin{align*}
  \rho(s)\rho(t)&=e'\delta_x\cdot e''\delta_y\\
  &=\0_x(\0\m_x(e')e'')w_{x,y}\delta_{xy}\\
  &=\0_x(\0\m_x(e')e'')w_{x,y}\delta_1\cdot\0_x(\0\m_x(e')e'')\delta_{xy}\\
  &=i(\0_x(\0\m_x(e')e'')w_{x,y})\rho(e'\delta_x\cdot e''\delta_y)\\
  &=i(\0_x(\0\m_x(e')e'')w_{x,y})\rho(st).
 \end{align*}
 Here in the first line $e'\delta_x\cdot e''\delta_y$ is the product in $A*_\Theta G$, and in the forth line the same two factors are multiplied in $E(A)*_\theta G=S$. Since $st=\0_x(\0\m_x(e')e'')\delta_{xy}$, it follows from Corollary~\ref{cor-sdelta_x(sdelta_x)^(-1)} that $\alpha(stt\m s\m)=\0_x(\0\m_x(e')e'')$, so 
 $$
  \0_x(\0\m_x(e')e'')w_{x,y}=\alpha(stt\m s\m)w_{x,y}=\alpha(stt\m s\m)w_{\kappa(s),\kappa(t)}
 $$
 belongs to $A_{\alpha(stt\m s\m)}$ and hence coincides with $f(s,t)$. 
\end{proof}

Observe that whenever $\Theta$ is equivalent to $\Theta'=(\0',w')$, then $E(A)*_\0 G=E(A)*_{\0'} G$ by Remark~\ref{rem-eq-tw-pact-rest-to-E(A)}. So, $\Lambda'=(\alpha',\lambda',f')$ corresponding to $\Theta'$ is a twisted module structure on $A$ over the same $E$-unitary inverse semigroup.
\begin{prop}\label{prop-equiv-Thetas=>equiv-Lambdas}
 If $\Theta$ is equivalent to $\Theta'$, then $\Lambda$ is equivalent to $\Lambda'$.
\end{prop}
\begin{proof}
 We need to check (i)--(iii) of Proposition~\ref{prop-equiv-tw-S-mod}. Item (i) immediately follows from~\eqref{eq-alpha(edelta_1)}.
 
 Let $s=e\delta_x\in S$. Then, $\alpha(ss\m)=e\in\cD_{\kappa(s)}$ and $\alpha(s\m s)=\0\m_x(e)\in\cD_{\kappa(s)\m}$ by Corollary~\ref{cor-sdelta_x(sdelta_x)^(-1)}. Set 
\begin{align}\label{eq-g(s)=eps_x-alpha(ss^(-1))}
 g(s)=\alpha(ss\m)\e_{\kappa(s)},
\end{align}
where $\e_x\in\cU{\cM{\cD_x}}$ determines the equivalence of $\Theta$ and $\Theta'$. Using (ii) of Definition~\ref{defn-equiv-tw-pact}, \eqref{eq-lambda_s(a)}, \eqref{eq-g(s)=eps_x-alpha(ss^(-1))}, \eqref{eq-lambda_s(e_lambda_s)}, Remark~\ref{rem-M(S)-for-comm-S} and Lemma~\ref{lem-(ws)^(-1)}, one has
$$
 \lambda'_s(a)=\0'_{\kappa(s)}(\alpha(s\m s)a)=\e_{\kappa(s)}\0_{\kappa(s)}(\alpha(s\m s)a)\e\m_{\kappa(s)}=g(s)\lambda_s(a)g(s)\m,
$$
which is (ii) of Proposition~\ref{prop-equiv-tw-S-mod}.

According to~\eqref{eq-lambda_s(a)}--\eqref{eq-g(s)=eps_x-alpha(ss^(-1))}, (ii) of Definition~\ref{defn-twisted_S-mod} and Remarks~\ref{rem-M(S)-for-comm-S} and~\ref{rem-eq-tw-pact-rest-to-E(A)}
\begin{align*}
 f'(s,t)g(st)&=\alpha(stt\m s\m)w'_{\kappa(s),\kappa(t)}\alpha(stt\m s\m)\e_{\kappa(st)}\\
 &=\alpha(stt\m s\m)w'_{\kappa(s),\kappa(t)}\e_{\kappa(st)}\\
 &=\lambda_s(\alpha(tt\m))w'_{\kappa(s),\kappa(t)}\e_{\kappa(st)}\\
 &=\0_{\kappa(s)}(\alpha(s\m stt\m))w'_{\kappa(s),\kappa(t)}\e_{\kappa(st)}\\
&={\0}'_{\kappa(s)}(\alpha(s\m stt\m))w'_{\kappa(s),\kappa(t)}\e_{\kappa(st)}.
\end{align*}
Here $\alpha(s\m stt\m)=\alpha(s\m s)\alpha(tt\m)\in\cD_{\kappa(s)\m}\cD_{\kappa(t)}$. Using the same facts again, we see that
\begin{align*}
 g(s)\lambda_s(g(t))f(s,t)&=\alpha(ss\m)\e_{\kappa(s)}\lambda_s(\alpha(tt\m)\e_{\kappa(t)})\alpha(stt\m s\m)w_{\kappa(s),\kappa(t)}\\
 &=\lambda_s(\alpha(s\m s))\e_{\kappa(s)}\lambda_s(\alpha(tt\m)\e_{\kappa(t)})\lambda_s(\alpha(tt\m))w_{\kappa(s),\kappa(t)}\\
 &=\e_{\kappa(s)}\lambda_s(\alpha(s\m stt\m)\e_{\kappa(t)})w_{\kappa(s),\kappa(t)}\\
 &=\e_{\kappa(s)}\0_{\kappa(s)}(\alpha(s\m stt\m)\e_{\kappa(t)})w_{\kappa(s),\kappa(t)}.
\end{align*}
Now (iii) of Proposition~\ref{prop-equiv-tw-S-mod} follows from (iii) of Definition~\ref{defn-equiv-tw-pact}.
\end{proof}

 \section{The correspondence between twisted partial actions and Sieben twisted modules}\label{sec-Theta<->Lambda}
\begin{prop}\label{prop-Lambda^(Theta^Lambda)}
 Let $\Lambda=(\alpha,\lambda,f)$ be a Sieben twisted $S$-module structure on $A$, $\Theta=(\theta,w)$ the twisted partial action of $\cG S$ on $A$ as in Proposition~\ref{prop-A*_Lambda-S-cong-A*_Theta-G(S)} and $\Lambda'=(\alpha',\lambda',f')$ the corresponding twisted $E(A)*_\0\cG S$-module structure on $A$. Then there exists an isomorphism $\nu:S\to E(A)*_\0\cG S$, such that $\Lambda=\Lambda'\circ\nu$ in the sense of~\eqref{Lambda=Lambda'-circ-nu}.
\end{prop}
\begin{proof}
 By Proposition~\ref{prop-A*_Lambda-S-cong-A*_Theta-G(S)} the extensions $A\overset{i}{\to}A*_\Lambda S\overset{j}{\to}\cG S$ and $A\overset{i'}{\to}A*_\Theta\cG S\overset{j'}{\to}\cG S$ are equivalent. Recall that here $i(a)=a\delta_{\alpha\m(aa\m)}$, $j(a\delta_s)=\sigma^\natural(s)$, $i'(a)=a\delta_1$, $j'(a\delta_x)=x$. Moreover, $\mu:A*_\Lambda S\to A*_\Theta\cG S$, $\mu(a\delta_s)=a\delta_{\sigma^\natural(s)}$, is an isomorphism defining the equivalence. 
 
 Representing $A*_\Lambda S\overset{j}{\to}\cG S$ as $A*_\Lambda S\overset{\pi}{\to}S\overset{\kappa}{\to}\cG S$ and $A*_\Theta\cG S\overset{j'}{\to}\cG S$ as $A*_\Theta\cG S\overset{\pi'}{\to}E(A)*_\0\cG S\overset{\kappa'}{\to}\cG S$ in accordance with Proposition~\ref{prop-ext-A-G-to-ext-A-S} (see also  Proposition~\ref{prop-A*_Theta-G-is-ext}), one finds by Proposition~\ref{prop-eq-ext-A-G-to-eq-ext-A-S} an isomorphism $\nu:S\to E(A)*_\0\cG S$, such that  
 $$
      	\begin{tikzpicture}[node distance=1.5cm, auto]
      		\node (A) {$A$};
      		\node (U) [right = 2cm of A] {$A*_\Lambda S$};
      		\node (S) [right = 2cm of U] {$S$};
      		\node (A') [below of=A]{$A$};
      		\node (U') [below of=U] {$A*_\Theta\cG S$};
      		\node (S') [below of=S] {$E(A)*_\0\cG S$};
      		\draw[->] (A) to node {$i$} (U);
      		\draw[->] (U) to node {$\pi$} (S);
      		\draw[->] (A') to node {$i'$} (U');
      		\draw[->] (U') to node {$\pi'$} (S');
      		\draw[-,double distance=2pt] (A) to node {} (A');
      		\draw[->] (U) to node {$\mu$} (U');
      		\draw[->] (S) to node {$\nu$} (S');
      	\end{tikzpicture}
 $$
 commutes. Here $\pi(a\delta_s)=s$, $\pi'(a\delta_x)=aa\m\delta_x$ and $\nu(s)=\alpha(ss\m)\delta_{\sigma^\natural(s)}$.  
 
 Let $\rho$ and $\rho'$ denote the standard transversals of $\pi$ and $\pi'$, respectively, i.\,e. $\rho(s)=\alpha(ss\m)\delta_s$ and $\rho'(e\delta_x)=e\delta_x$. Observe that $\mu\circ\rho=\rho'\circ\nu$. Indeed,
 \begin{align*}
  \mu\circ\rho(s)&=\mu(\alpha(ss\m)\delta_s)=\alpha(ss\m)\delta_{\sigma^\natural(s)},\\
  \rho'\circ\nu(s)&=\rho'(\alpha(ss\m)\delta_{\sigma^\natural(s)})=\alpha(ss\m)\delta_{\sigma^\natural(s)}.
 \end{align*}
 The twisted $S$-module structure coming from $(A*_\Lambda S,\rho)$ is $\Lambda$ itself thanks to Corollary~\ref{cor-Lambda-for-A*_Lambda-S}. The pair $(A*_\Theta\cG S,\rho')$ induces $\Lambda'$ by construction (see Section~\ref{sec-Lambda^Theta}). Using Corollary~\ref{cor-equiv-ext-by-isomorphic-S} one concludes that $\Lambda=\Lambda'\circ\nu$.
\end{proof}

\begin{prop}\label{prop-Theta^(Lambda^Theta)}
 Let $\Theta=(\0,w)$ be a twisted partial action of $G$ on $A$, $\Lambda=(\alpha,\lambda,f)$ the corresponding twisted $E(A)*_\0 G$-module structure on $A$ and $\Theta'=(\0',w')$ the twisted partial action of $\cG{E(A)*_\0 G}$ on $A$ coming from $\Lambda$ as in Proposition~\ref{prop-A*_Lambda-S-cong-A*_Theta-G(S)}. Then there is an isomorphism $\nu:G\to\cG{E(A)*_\0 G}$, such that $\Theta=\Theta'\circ\nu$ (that is $\0_x=\0'_{\nu(x)}$ and $w_{x,y}=w'_{\nu(x),\nu(y)}$).
\end{prop}
\begin{proof}
 Since $\Lambda$ is the twisted $E(A)*_\0 G$-module structure defined by the extension $A\overset{i}{\to}A*_\Theta G\overset{j}{\to}E(A)*_\0 G$ and the transversal $\rho(e\delta_x)=e\delta_x$ of $j$, by Remark~\ref{rem-A*_Lambda-S-equiv-U} the crossed products $A*_\Theta G$ and $A*_\Lambda(E(A)*_\0 G)$ are isomorphic. An arbitrary element of $A*_\Lambda(E(A)*_\0 G)$ has the form $a\delta_{e\delta_x}$, where $aa\m=\alpha(e\delta_x(e\delta_x)\m)=\alpha(e\delta_1)=e$, that is $a\delta_{e\delta_x}=a\delta_{aa\m\delta_x}$. According to Remark~\ref{rem-A*_Lambda-S-equiv-U} the isomorphism $A*_\Lambda(E(A)*_\0 G)\to A*_\Theta G$ identifies $a\delta_{aa\m\delta_x}$ with
 $$
  i(a)\rho(aa\m\delta_x)=a\delta_1\cdot aa\m\delta_x=aa\m a\delta_x=a\delta_x.
 $$
 Denote by $\mu$ the inverse isomorphism $A*_\Theta G\to A*_\Lambda(E(A)*_\0 G)$, $\mu(a\delta_x)=a\delta_{aa\m\delta_x}$.
 
 By Proposition~\ref{prop-A*_Lambda-S-cong-A*_Theta-G(S)} the semigroup $A*_\Lambda(E(A)*_\0 G)$ is isomorphic to $A*_{\Theta'}\cG{E(A)*_\0 G}$ by means of $\mu'$ mapping $a\delta_{aa\m\delta_x}\in A*_\Lambda(E(A)*_\0 G)$ to $a\delta'_{\sigma^\natural(aa\m\delta_x)}\in A*_{\Theta'}\cG{E(A)*_\0 G}$. Then the composition $\mu'\circ\mu$ is an isomorphism $A*_\Theta G\to A*_{\Theta'}\cG{E(A)*_\0 G}$, such that
 $$
  \mu'\circ\mu(a\delta_x)=a\delta'_{\sigma^\natural(aa\m\delta_x)}.
 $$
 
 Observe by Proposition~\ref{prop-A*_Theta-G-is-ext} that the map $aa\m\delta_x\mapsto x$ is an epimorphism $E(A)*_\0 G\to G$ whose kernel is $\sigma$. Hence, there is an isomorphism $\nu:G\to\cG{E(A)*_\0 G}$, identifying $x\in G$ with $\sigma^\natural(e\delta_x)\in\cG{E(A)*_\0 G}$, where $e$ is an arbitrary idempotent of $\cD_x$. It induces the isomorphism $\mu'':A*_{\Theta'}\cG{E(A)*_\0 G}\to A*_{\Theta'\circ\nu}G$, which maps $a\delta_{\sigma^\natural(e\delta_x)}$ to $a\delta''_x$. Then $\mu''\circ\mu'\circ\mu$ is an isomorphism between the crossed products $A*_\Theta G$ and $A*_{\Theta'\circ\nu}G$ sending $a\delta_x$ to $a\delta''_x$. By Lemma~\ref{lem-S*_Theta-G-cong-S*_Theta'-G} one has $\Theta=\Theta'\circ\nu$.
\end{proof}

\begin{thrm}\label{thrm-Theta<->Lambda}
 Let $A$ be a semilattice of groups. Up to identification of isomorphic groups and semigroups, there is a one-to-one correspondence between twisted partial actions of groups on $A$ and Sieben twisted module structures over $E$-unitary inverse semigroups on $A$. Moreover, equivalent twisted partial actions correspond to equivalent Sieben twisted modules.
\end{thrm}
\begin{proof}
 This follows from Propositions~\ref{prop-equiv-Lambdas=>equiv-Thetas},~\ref{prop-equiv-Thetas=>equiv-Lambdas},~\ref{prop-Lambda^(Theta^Lambda)} and~\ref{prop-Theta^(Lambda^Theta)}
\end{proof}

\section*{Acknowledgements}
We thank the anonymous referee for his/her careful reading of the manuscript and many helpful comments.

\bibliography{bibl-pact}{}
\bibliographystyle{acm}

\end{document}